\documentclass[11pt, a4paper, twoside]{amsart}

\usepackage{todonotes}

\usepackage{amsmath}
\usepackage{amsfonts}
\usepackage{amssymb}
\usepackage{latexsym}

\usepackage{hyperref}
\hypersetup{
colorlinks=true,  citecolor=magenta, linkcolor=cyan, urlcolor=blue}

\usepackage{mathtools}
\mathtoolsset{showonlyrefs}

\usepackage{graphicx}
\usepackage{xcolor}
\usepackage{graphicx} 

\usepackage{enumitem} 

\allowdisplaybreaks

\theoremstyle{change}
\newtheorem{theo}{Theorem}[section]
\newtheorem*{theo*}{Theorem}
\newtheorem{lemma}[theo]{Lemma}
\newtheorem{corollary}[theo]{Corollary}

\newtheorem{conj}[theo]{Conjecture}
\newtheorem*{conj*}{Conjecture}
\newtheorem{rem}[theo]{Remark}

\newtheorem{proposition}[theo]{Proposition}  

\numberwithin{equation}{section}

\newcommand{\ir}{\mathrm{r}} 
\newcommand{\cR}{\mathrm{R}} 

\newcommand{\inte}{\mathrm{int}\,}  
\newcommand{\conv}{\mathrm{conv}\,} 

\DeclareMathOperator{\lin}{lin}

\newcommand{\vol}{\mathrm{vol}} 

\newcommand{\grass}{\mathcal{G}} 
\newcommand{\Lat}{\mathcal{L}} 

\newcommand{\Kn}{{\mathcal K}^n} 

\newcommand{\R}{\mathbb{R}} 

\newcommand{\Z}{\mathbb{Z}} 

\newcommand{\disuni}{\mathbin{\setbox0\hbox{$\bigcup$}\rlap{\copy0}\raise.3%
  \ht0\hbox to \wd0{\hfil$\cdot$\hfil}}}
  
\newcommand{\ip}[2]{\left\langle #1,#2\right\rangle}

\newcommand{\dual}[1]{{[#1]}^\star}
\newcommand{\ov}{\overline}

\newcommand{\ban}{B_n}

\newcommand{\trans}{\intercal}

\newcommand{\va}{{\boldsymbol a}}
\newcommand{\vb}{{\boldsymbol b}}

\newcommand{\ve}{{\boldsymbol e}}

\newcommand{\vr}{{\boldsymbol r}}
\newcommand{\vs}{{\boldsymbol s}}
\newcommand{\vt}{{\boldsymbol t}}

\newcommand{\vu}{{\boldsymbol u}}
\newcommand{\vv}{{\boldsymbol v}}
\newcommand{\vw}{{\boldsymbol w}}
\newcommand{\vx}{{\boldsymbol x}}
\newcommand{\vy}{{\boldsymbol y}}
\newcommand{\vz}{{\boldsymbol z}}

\newcommand{\vnull}{{\boldsymbol 0}}

\newcommand{\suk}{\mathrm{h}}

\newcommand{\vbeta}{{\boldsymbol \beta}}

\begin{document}

\title{Packing minima and lattice points in convex bodies} 

\author{Martin Henk}
\address{Institut f\"ur Mathematik\\
  Technische Universit\"at Berlin\\
  Stra\ss e des 17. Juni 136\\
  D-10623 Berlin\\
  Germany}
\email{henk@math.tu-berlin.de}

\author{Matthias Schymura}
\address{Institut f\"ur Mathematik\\
  BTU Cottbus-Senftenberg\\
  Platz der Deutschen Einheit 1\\
  D-03046 Cottbus\\
  Germany}
\email{schymura@b-tu.de}

\author{Fei Xue}
\address{Institut f\"ur Mathematik\\
  Technische Universit\"at Berlin\\
  Stra\ss e des 17. Juni 136\\
  D-10623 Berlin\\
  Germany}
\email{xue@math.tu-berlin.de}

\date{\today}


\begin{abstract}
Motivated by long-standing conjectures on the discretization of classical inequalities in the Geometry of Numbers, we investigate a new set of parameters, which we call \emph{packing minima}, associated to a convex body $K$ and a lattice $\Lambda$.
These numbers interpolate between the successive minima of $K$ and the inverse of the successive minima of the polar body of $K$, and can be understood as packing counterparts to the covering minima of Kannan \& Lov\'{a}sz (1988).

As our main results, we prove sharp inequalities that relate the volume and the number of lattice points in~$K$ to the sequence of packing minima.
Moreover, we extend classical transference bounds and discuss a
natural class of examples in detail.
\end{abstract}

\maketitle

\section{Introduction}

%
%
%
%
%
%
%
%
%
%
%
%
%
%

Let $\Kn$ be the set of all \emph{convex bodies}, that is, compact convex sets in the $n$-dimensional Euclidean space $\R^n$ with non-empty interior,
and let $\Lat^n$ be the set of all $n$-dimensional \emph{lattices}
$\Lambda \subseteq \R^n$. Here, a lattice $\Lambda\in\Lat^n$ is understood
as the linear image of the standard lattice $\Z^n$, that is, there exists
an invertible matrix $B \in \R^{n \times n}$ such that
$\Lambda=B\Z^n$. The matrix~$B$, or more precisely the set of its
column vectors, is called a  basis of~$\Lambda$ and $|\det(B)|$ is called the determinant of $\Lambda$, which will be denoted by $\det(\Lambda)$.  
 
One of the classical functionals in the Geometry of Numbers is the {\em covering radius} $\mu(K,\Lambda)$ defined for $K \in \Kn$ and
$\Lambda\in\Lat^n$ as the smallest positive dilation factor $\mu$
such that $\Lambda$ is a \emph{covering lattice} of $\mu\,K$, that is,  
\begin{equation}
  \label{eq:coveringradius}
    \mu(K,\Lambda)=\min\{\mu>0 : \Lambda+\mu\,K=\R^n\}.  
\end{equation}
In their studies of lattice point free convex bodies, Kannan \&
Lov{\'a}sz \cite{kannanlovasz1988covering} embedded this functional 
into a series of functionals, called {\em covering minima}.
In order to define them, let $\grass(\R^n,k)$ be the Grassmannian of all
$k$-dimensional linear subspaces of~$\R^n$, and for a given lattice
$\Lambda\in\Lat^n$ let
\begin{equation*}
         \grass(\Lambda,k) = \left\{L\in\grass(\R^n,k) : \dim(L\cap\Lambda)=k\right\}
\end{equation*}
be the set of all $k$-dimensional linear {\em lattice planes} of
$\Lambda$.
As usual, the dimension $\dim(S)$ of a set $S \subseteq \R^n$ is defined as the dimension of the affine hull of~$S$.
Now, the \emph{$i$th covering minimum} $\mu_i(K,\Lambda)$ is defined as
the smallest positive number $\mu$ such that for every $L \in \grass(\R^n,n-i)$ the lattice $\Lambda$ is a covering lattice of $\mu\,K+L$, that is,  
\begin{equation}
  \mu_i(K,\Lambda)=\inf\{\mu>0 :
  (\Lambda+\mu\,K+L)=\R^n\text{ for all }L\in\grass(\R^n,n-i)\},
\label{eq:defcoveringminima}  
\end{equation}
for $1\leq i\leq n$.  For $i=n$, we retrieve the covering radius $\mu_n(K,\Lambda)=\mu(K,\Lambda)$,
and as it was pointed out in \cite[Lem.~(2.3)]{kannanlovasz1988covering},  the infimum in the above definition is attained for a certain lattice
plane~$L$. Hence, we may also write
\begin{equation}
  \mu_i(K,\Lambda)=\max\left\{\mu(K|L^\perp,\Lambda|L^\perp) :
  L\in\grass(\Lambda,n-i)\right\},
  \label{eq:def-covering-minima-projection}
\end{equation}
where $L^\perp$ denotes the orthogonal complement of $L$, and $\cdot|L^\perp$ the orthogonal projection onto it.
Note that since $L$ is a lattice plane, the projection $\Lambda|L^\perp$ is indeed again a lattice.

The counterpart to the covering radius in the Geometry of Numbers is the
{\em packing radius} $\rho(K,\Lambda)$, which is defined as the maximal $\rho>0$ such that a given lattice $\Lambda$ is a {\em packing lattice} of $\rho\,K$, that is, we have
$\inte(\vx+\rho\,K)\cap\inte(\vy+\rho\,K)=\emptyset$, for all
$\vx,\vy\in \Lambda$ with $\vx \neq \vy$, and where $\inte(\cdot)$ denotes the topological interior. Since
a lattice is an additive group this is equivalent to saying that
\begin{equation*} 
\rho(K,\Lambda)=\max\{\rho>0 :
\inte(\rho(K-K))\cap\Lambda=\{\vnull\}\}.
\end{equation*}  
The set $K-K = \{\vx-\vy:\vx,\vy\in K\}$ is called the \emph{difference body} of~$K$, and it clearly belongs to the set of \emph{origin-symmetric} convex bodies, that is, those $M \in \Kn$ satisfying $M = -M$.
For the sake of brevity and since we use the difference body oftentimes in the sequel, we denote it by $K_c = K-K$. We may thus also write
\begin{equation*}
  \rho(K,\Lambda)=\lambda(K_c,\Lambda):=\min\{\lambda>0 : \lambda\,K_c\cap\Lambda\ne\{\vnull\}\},
\end{equation*}
where $\lambda(K_c,\Lambda)$  is the length of a
shortest non-trivial lattice vector in~$\Lambda$, and that length
being measured with respect to the norm $\|\vx\|_{K_c} = \min\{r \geq
0 : \vx \in r\, K_c\}$ induced by $K_c$.

Now, in analogy to \eqref{eq:defcoveringminima},
we define the \emph{$i$th packing minimum} as 
\begin{equation*}
  \rho_i(K,\Lambda)=\inf\{\rho>0:
  (\rho\, K_c+L)\cap\Lambda \ne  L \cap \Lambda \text{ for all }L\in\grass(\R^n,n-i)\},
\end{equation*}
for $1\leq i\leq n$. 
 As we will see in Lemma \ref{lem:rho-i-attained}, and in analogy with
\eqref{eq:def-covering-minima-projection}, the infimum is attained
at a certain lattice plane, and we may write 
\begin{align}
  \rho_i(K,\Lambda)=\max\left\{\lambda(K_c|L^\perp,\Lambda|L^\perp):L\in\grass(\Lambda,n-i)\right\}.\label{eq:def-packing-minima}
\end{align}
In words, $\rho_i(K,\Lambda)$ is the smallest number $\rho > 0$ such that every projection of $\rho\, K_c$ along an $(n-i)$-dimensional lattice plane~$L$ of $\Lambda$ contains a non-trivial lattice point of the $i$-dimensional lattice~$\Lambda|L^\perp$. In particular,
 \begin{equation} 
   \rho_n(K,\Lambda)=\rho(K,\Lambda) \quad\text{ and }\quad \rho_1(K,\Lambda)=\mu_1(K,\Lambda).
 \label{eq:rho-1-n}   
\end{equation} 
In order to understand the geometric meaning of the packing minima, we need some more definitions. First, for an origin-symmetric convex body $M\in\Kn$, the functional $\lambda(M,\Lambda)$ was embedded by Minkowski in a series of functionals which are known as Minkowski's successive minima. Here, the \emph{$i$th successive minimum} $\lambda_i(M,\Lambda)$ is defined as the smallest dilation factor $\lambda>0$ such that $\lambda\,M$ contains at least~$i$ linearly independent lattice points of $\Lambda$, that is, for $1\leq i\leq n$   
\begin{equation*}
  \lambda_i(M,\Lambda)=\min\{\lambda>0:
  \dim(\lambda\,M\cap\Lambda)\geq i\}. 
\end{equation*}
Hence, we retrieve $\lambda_1(M,\Lambda)=\lambda(M,\Lambda)$.
Furthermore, we need the {\em polar body} $\dual{K}$ of $K\in\Kn$ defined by
 \begin{equation*}
   \dual{K}=\{\vy\in\R^n : \ip{\vx}{\vy}\leq 1, \text{ for all }
   \vx\in K\}, 
 \end{equation*}
 where $\ip{\vx}{\vy}$ denotes the standard inner product on
 $\R^n$. The counterpart for lattices is the so-called {\em polar lattice} $\dual{\Lambda}$ of $\Lambda\in\Lat^n$ given by
\begin{equation*}
   \dual{\Lambda}=\{\va\in\R^n : \ip{\vb}{\va}\in\Z, \text{ for all }
   \vb\in \Lambda\}. 
 \end{equation*}
The first successive minimum $\lambda_1(\dual{K_c},\dual{\Lambda})$ in
this polar setting is also known as the {\em lattice width} of $K$ (with respect to~$\Lambda$):
\begin{equation*}
  \lambda_1(\dual{K_c},\dual{\Lambda})=\min_{\vv \in \dual{\Lambda} \setminus\{\vnull\}}\left(\max_{\vx\in
  K}\ip{\vv}{\vx}-\min_{\vy\in
  K}\ip{\vv}{\vy}\right).
\end{equation*}
In \cite[Prop.~(2.4)]{kannanlovasz1988covering} it is shown that
\begin{align}
  \mu_1(K,\Lambda)=\frac{1}{\lambda_1(\dual{K_c},\dual{\Lambda})},\label{eq:mu1-lat-width} 
\end{align}
and so, in view of~\eqref{eq:rho-1-n}, we also have $\rho_1(K,\Lambda)=\lambda_1(\dual{K_c},\dual{\Lambda})^{-1}$.       
%
For more information on the Geometry of Numbers, the theory of convex bodies, lattices and polarity, we refer to the excellent textbooks~\cite{gruber2007convex,gruberlekker1987geometry,martinet2003perfect}.

The discussion above shows that the packing minima interpolate between
$\rho_1(K,\Lambda)=\lambda_1(\dual{K_c},\dual{\Lambda})^{-1}$ and the
first successive minimum $\rho_n(K,\Lambda) = \lambda_1(K_c,\Lambda)$
of~$K_c$. In fact, this holds more generally as 
\begin{proposition} Let $K\in\Kn$ and $\Lambda\in\Lat^n$. Then for
  $1\leq i\leq n$  
\begin{align}
\frac{1}{\lambda_i(\dual{K_c},\dual{\Lambda})} \leq \rho_i(K,\Lambda) \leq \lambda_{n-i+1}(K_c,\Lambda),
\end{align}
Moreover, the upper bound is always attained for $i=n$ and the lower
bound for $i=1$.
\label{prop:succmin-bounds-intro}
\end{proposition}   
%
%
With these bounds at hand, we can now discuss a first example and determine the packing minima of axis parallel boxes.
To this end, for $\vr=(r_1,\ldots,r_n) \in \R^n$ with $0 < r_1 \leq \ldots \leq r_n$, we write $C(\vr)=[-r_1,r_1]\times\ldots\times[-r_n,r_n]$.
Then, for $1 \leq i \leq n$, we have
\begin{align}
\frac{1}{\lambda_i(\dual{2C(\vr)},\Z^n)} = \rho_i(C(\vr),\Z^n) = \lambda_{n-i+1}(2C(\vr),\Z^n) = \frac{1}{2r_i},\label{eq:ex-boxes}
\end{align}
and in particular,
\begin{align}
\rho_1(C(\vr),\Z^n) \geq \rho_2(C(\vr),\Z^n) \geq \ldots \geq \rho_n(C(\vr),\Z^n).\label{eq:ex-boxes-monotonicity}
\end{align}
In view of Proposition \ref{prop:succmin-bounds-intro} and the homogeneity of the involved functionals, it suffices to compute the
respective successive minima of $C(\vr)$ and $\dual{C(\vr)} =
\conv\!\big(\big\{\pm \frac{1}{r_1} \ve_1,\ldots,\pm \frac{1}{r_n}
  \ve_n\big\}\big)$ and see that they agree. Here, $\ve_i$ are the
canonical unit vectors.
Since $r_1 \leq \ldots \leq r_n$, we clearly have $\lambda_i(C(\vr),\Z^n)=\frac{1}{r_{n-i+1}}$, for $1 \leq i \leq n$.
Similarly, one finds that $\lambda_i(\dual{C(\vr)},\Z^n) = r_i$, and~\eqref{eq:ex-boxes} follows.

Of course, since for the boxes $C(\vr)$, the packing minima agree with the successive minima we do not obtain new information for these examples.
However, in Section~\ref{sect:reg-simplex-lattice}, we will study more involved examples that explicitly show that 
the inequalities in Proposition \ref{prop:succmin-bounds-intro} can be strict (see Corollary~\ref{cor:Lambda-ab-bounds}).

From the definition of the successive minima and the
covering minima it is clear that these two sequences are
increasing. 
In view of~\eqref{eq:ex-boxes-monotonicity}, one might suspect that
the packing minima form a decreasing sequence.
This is however not true in general, as otherwise we would always have $\rho_1(K,\Lambda)
\geq\rho_n(K,\Lambda)$ which is equivalent to
$\lambda_1(\dual{K_c},\dual{\Lambda}) \cdot \lambda_1(K_c,\Lambda) \leq 1$.
But the order of magnitude of this latter product can be linear in the dimension~$n$.
More precisely, Conway \& Thompson showed that there are self-dual lattices $\Lambda \in \Lat^n$, meaning $\dual{\Lambda}=\Lambda$, such that $\lambda_1(B_n,\Lambda)^2 \in \Omega(n)$ (cf.~\cite[Ch.~2, \S9]{milnorhusemoller1973symmetric}). 

Due to their close relation to the successive minima of $K_c$ and $\dual{K_c}$, the packing
minima provide a new tool to improve known
results involving successive minima, and this was, indeed, our main
motivation to investigate these functionals. We start with  
\emph{Minkowski's 2nd Theorem on successive minima} that establishes a
relation between the successive minima and the \emph{volume} $\vol(K)$, that is, the $n$-dimensional Lebesgue measure, of a convex body: 
Let $K\in \Kn$ and $\Lambda\in\Lat^n$. Then
\begin{equation}
\frac{1}{n!} \prod_{i=1}^n \frac{1}{\lambda_i(K_c,\Lambda)} \leq
\frac{\vol(K)}{\det(\Lambda)} \leq \prod_{i=1}^n
\frac{1}{\lambda_i(K_c,\Lambda)}.
\label{eq:minkowskisecond} 
\end{equation} 
Minkowski originally proved these bounds for origin-symmetric convex bodies (see~\cite{henk2002successive} for a modern exposition of Minkowski's original ideas, and for more background information).
However, an application of the well-known Brunn-Minkowski inequality allows to extend the upper bound to arbitrary $K \in \Kn$ as stated (cf.~\cite[p.~59]{gruberlekker1987geometry} or~\cite[(1.2)]{henkhenzecifre2016variations}).
The lower bound can be much easier derived; we refer to the discussion
in~\cite[p.~62]{gruberlekker1987geometry} in combination with the
remarks after~\cite[(1.2)]{henkhenzecifre2016variations}.

Due to Proposition \ref{prop:succmin-bounds-intro}, the upper bound in
\eqref{eq:minkowskisecond} is also valid for the packing minima, and
in fact, we also have a lower bound in terms of the packing minima. 
\begin{theo}
\label{thm:minkowskisecondrho}
Let $K\in\Kn$ and $\Lambda\in\Lat^n$. Then 
\begin{equation}
\frac{1}{n!} \prod_{i=1}^n \frac{1}{\rho_i(K,\Lambda)} \leq
\frac{\vol(K)}{\det(\Lambda)} \leq \prod_{i=1}^n
\frac{1}{\rho_i(K,\Lambda)}.  
\end{equation}
\end{theo}

\noindent Observe, that by Proposition~\ref{prop:succmin-bounds-intro} the lower bound improves upon~\eqref{eq:minkowskisecond}.
Moreover, both bounds are best possible as shown by $C(\vr)$ for the upper bound and by $\dual{C(\vr)}$ for the lower bound.

Regarding lattice points and successive minima we have (cf.~\cite[Thm.~2.1]{betkehenkwills1993successive})
\begin{align}
  \#(K\cap\Lambda)\leq \left\lfloor
    \frac{1}{\lambda_1(K_c,\Lambda)}+1\right\rfloor^n = \left\lfloor
    \frac{1}{\rho_n(K,\Lambda)}+1\right\rfloor^n.
\label{eq:disc-lambda1-bound}  
\end{align}
In the same paper it has been conjectured that, for origin-symmetric
convex bodies, this inequality holds in a much stronger fashion. 
Later, in~\cite{malikiosis2010adiscrete}, the conjecture was extended
to arbitrary convex bodies  $K \in \Kn$.
\begin{conj*}[{\cite{betkehenkwills1993successive,malikiosis2010adiscrete}}]
\label{conj:discrete-2nd-minkowski}
Let $K \in \Kn$ and $\Lambda \in \Lat^n$.
Then,
\begin{align}
\#(K\cap\Lambda) \leq \prod_{i=1}^n
      \left\lfloor\frac{1}{\lambda_i(K_c,\Lambda)}+1 \right\rfloor.\label{eq:conj-discrete-2nd-Mink}
\end{align}
\end{conj*}

\noindent The importance of this conjecture comes from the fact that it
discretizes the upper bound in Minkowski's 2nd Theorem
\eqref{eq:minkowskisecond}  in view of the identity
\begin{align}
\frac{\vol(K)}{\det(\Lambda)} = \lim_{s \to 0} s^n \#(K \cap s \Lambda)= \lim_{r \to \infty} r^{-n} \#(rK \cap \Lambda)\label{eq:vol-limit-numlatpoints}
\end{align}
and the homogeneity of the successive minima (see~\cite[Prop.~2.2]{betkehenkwills1993successive} for details).
Moreover, a (combinatorial) proof of the conjecture
would lead to new insights to Minkowski's deep and fundamental result.
For planar origin-symmetric convex bodies, the
conjecture holds as shown in~\cite{betkehenkwills1993successive}.
Malikiosis~\cite[Sect.~3.2]{malikiosis2010adiscrete}  proved
it for
arbitrary $K \in \Kn$ whenever $n \leq 3$, and
in~\cite[Thm.~3.1.2]{malikiosis2010adiscrete} he verified
the bound in general dimension up to the multiplicative
constant $(4/e)\sqrt{3}^{n-1}$.

While this discretization of Minkowski's classical result turns out to be quite intricate, it is easy to derive a best possible upper bound on the number of lattice points in terms of the successive minima of the polar body (cf.~Proposition~\ref{prop:second-minkowski-polarlattice}):
\begin{equation} 
  \#(K\cap\Lambda) \leq \prod_{i=1}^n
      \left\lfloor \lambda_i(\dual{K_c},\dual{\Lambda})+1
      \right\rfloor.
\label{eq:lattice-upper-dual}      
\end{equation}  
Our main result establishes the analogous upper bound in terms of the
packing minima.
Taking the inequalities in Proposition \ref{prop:succmin-bounds-intro} into account, this improves upon~\eqref{eq:lattice-upper-dual} and provides additional evidence for Conjecture~\eqref{eq:conj-discrete-2nd-Mink}.
\begin{theo} 
\label{thm:main}
Let $K\in\Kn$ and $\Lambda\in\Lat^n$. Then,
  \begin{equation*}
     \#(K\cap\Lambda) \leq \prod_{i=1}^n
      \left\lfloor\frac{1}{\rho_i(K,\Lambda)}+1 \right\rfloor, 
  \end{equation*}
  and this inequality cannot be improved in general.
  \label{theo:enumerator}
\end{theo}

Just as for the estimate on the volume, examples that attain this bound are the axis parallel boxes $C(\vr)$, where every $r_i$ is a positive integer.
Indeed, the number of lattice points equals $\#(C(\vr) \cap \Z^n) = \prod_{i=1}^n(2 r_i +1)$ and the packing minima are as in~\eqref{eq:ex-boxes}.
Although the lower bound on the volume in Theorem~\ref{thm:minkowskisecondrho} is not that deep, it seems to be quite challenging to find a discrete variant that lower bounds $\#(K\cap\Lambda)$ in terms of the~$\rho_i(K,\Lambda)$. 
A corresponding bound for the successive minima has been obtained in ~\cite[Cor.~2.1]{betkehenkwills1993successive}, who proved that
\[
\#(K \cap \Lambda) \geq \frac{1}{n!}\prod_{i=1}^n\left(\frac{1}{\lambda_i(K_c,\Lambda)}-1\right),
\]
for origin-symmetric convex bodies $K \in \Kn$ such that $\lambda_n(K_c,\Lambda) \leq 1$.
Based on the fact that $\rho_1(K,\Lambda) = 1 / \lambda_1(\dual{K_c},\dual{\Lambda})$ and $\rho_n(K,\Lambda) = \lambda_1(K_c,\Lambda)$, we obtain a satisfactory answer in the plane and prove

\begin{theo}
\label{thm:discrete-lower-bound-plane}
Let $K \in \mathcal{K}^2$, $\Lambda \in \Lat^2$ be such that $\rho_2(K,\Lambda) \leq \frac12$.
Then,
\[
\#(K \cap \Lambda) \geq \frac12\left\lfloor \frac{1}{\rho_1(K,\Lambda)}-2 \right\rfloor \left( \frac{1}{\rho_2(K,\Lambda)} - 2 \right) - 2.
\]
If $K$ is moreover origin-symmetric, then
\[
\#(K \cap \Lambda) \geq \left\lfloor \frac{1}{2\rho_1(K,\Lambda)} \right\rfloor \left( \frac{1}{\rho_2(K,\Lambda)} - 2 \right).
\]
\end{theo}
\noindent Approximating the volume of~$K$ by the number of lattice points in large dilates $r K$ (see~\eqref{eq:vol-limit-numlatpoints}) and looking at the lower bound in Theorem~\ref{thm:minkowskisecondrho}, we see that the factor $1/2$ is best possible.

We finally want to mention the related concept of a (general)
Korkine-Zolotarev reduced basis, as introduced in
\cite[Sect.~4]{kannanlovasz1988covering}: A basis $\vb_1,\dots,\vb_n$
of~$\Lambda$ is called \emph{Korkine-Zolotarev reduced} (with respect
to $K$) if, for $1\leq i\leq n$ and
$L_i=\lin\{\vb_1,\dots,\vb_{i-1}\} \in \grass(\Lambda,i-1)$, holds 
\begin{equation*}
           \|\vb_i|L_i^\perp\|_{K_c|L_i^\perp} =   \lambda(K_c|L_i^\perp,\Lambda|L_i^\perp).
\end{equation*}
In other words, $\vb_i|L_i^\perp$ is a shortest non-zero vector in the projected lattice $\Lambda|L_i^\perp$ with respect to the norm induced by $K_c|L_i^\perp$.
Hence, we have 
\[
\rho_{n-i+1}(K,\Lambda)\geq
\|\vb_{i}|L_{i}^\perp\|_{K_c|L_{i}^\perp},
\]
for each $1\leq i \leq n$.

The paper is organized as follows: In
Section~\ref{sect:transference-bounds} we study some basic properties
of the packing minima, in particular, relations to successive minima
and covering minima. The proof of the volume inequalities, that is, of Theorem~\ref{thm:minkowskisecondrho} is given in
Section~\ref{sect:volume-inequalities}, whereas Section~\ref{sect:main-result} deals with relations of the packing minima to the lattice point enumerator and the proofs of Theorem~\ref{theo:enumerator} and Theorem~\ref{thm:discrete-lower-bound-plane}.
In the last section, we study the packing minima for a specific family of lattices in order to show that they are different from the successive minima
and that they neither form a decreasing nor an increasing sequence in general (cf.~Proposition~\ref{prop:ex-ball-symmetric-lattice} and Corollary~\ref{cor:Lambda-ab-bounds}).

\section{Transference bounds for packing minima}
\label{sect:transference-bounds}

\noindent We first prove that the infimum in the definition of the packing minima is always attained by a lattice plane, and we use this fact without further reference below.

\begin{lemma}
\label{lem:rho-i-attained}
Let $K \in \Kn$ and $\Lambda \in \Lat^n$.
Then, for every $1 \leq i \leq n$, there exists a lattice plane $L \in \grass(\Lambda,n-i)$ such that
\[
  \begin{split} 
\rho_i(K,\Lambda) & = \min\{\rho>0:
(\rho\, K_c+L)\cap\Lambda \neq L \cap \Lambda\} \\ 
& = \lambda_1(K_c|L^\perp,\Lambda|L^\perp).
\end{split} 
\]
\end{lemma}

\begin{proof}
First, for  $L \in \grass(\R^n,n-i)$, we consider   
\begin{align*}
\bar\rho(L) &= \inf\{\rho>0 : (\rho\, K_c+L)\cap\Lambda \neq L \cap \Lambda\}\\
&= \inf\{\rho>0 : \rho\, (K_c|L^\perp)\cap\Lambda|L^\perp \neq \{\vnull\}\}.
\end{align*}
Every closed additive subgroup of~$\R^n$ is isomorphic to the cartesian product $\R^a \times \Z^b$, for some non-negative integers $a,b$ (cf.~\cite[Chapt.~VII]{bourbaki1966generaltop}).  
Therefore, if~$L$ is not a lattice plane, i.e., $L \notin
\grass(\Lambda,n-i)$, then the group~$\Lambda|L^\perp$ contains a
dense subset.
Since $K_c|L^\perp$ is full-dimensional in $L^\perp$, this means that $\bar\rho(L)=0$ for every $L\in
\grass(\R^n,n-i)\setminus \grass(\Lambda,n-i)$.
Hence
\begin{equation*}
  \begin{split}
  \rho_i(K,\Lambda) & = \sup\{\bar\rho(L) : L\in
  \grass(\Lambda,n-i)\}\\ & =
  \sup\{\lambda_1(K_c|L^\perp,\Lambda|L^\perp) : L \in \grass(\Lambda,n-i)\}.
  \end{split}
\end{equation*}
In order to show that the $\sup$ is indeed a $\max$, let 
$r\ban\subseteq K_c$ for some $r>0$, and we fix a lattice plane $\ov L\in
\grass(\Lambda,n-i)$ and set $\lambda := \lambda_1(K_c|\ov L^\perp,\Lambda|\ov L^\perp)$. 
By Minkowski's upper bound in \eqref{eq:minkowskisecond}, for any $L\in
\grass(\Lambda,n-i)$, we have the bound 
\begin{equation*}
  \lambda_1(K_c|L^\perp,\Lambda|L^\perp)\leq
  \left(\frac{\det(\Lambda|L^\perp)}{\vol_i(K_c|L^\perp)}\right)^\frac{1}{i}, 
\end{equation*}
where $\vol_i(\cdot)$ denotes the $i$-dimensional volume.
Because $\vol_i(K_c|L^\perp) \geq \vol_i(rB_n|L^\perp) = r^i\,\vol_i(B_i)=:\gamma(i,r)$ and
$\det(\Lambda|L^\perp)=\det(\Lambda)/\det(\Lambda\cap L)$ (cf.~\cite[Prop.~1.9.7]{martinet2003perfect}), we get the
upper bound
\begin{equation*}
  \lambda_1(K_c|L^\perp,\Lambda|L^\perp) \leq
  \left(\frac{1}{\gamma(i,r)}\frac{\det(\Lambda)}{\det(\Lambda\cap L)}\right)^\frac{1}{i}.
\end{equation*}
Thus, if
\begin{equation}
   \det(\Lambda\cap L) > \frac{\det(\Lambda)}{\gamma(i,r)\lambda^i}
\label{eq:obbound}
\end{equation}
then $\lambda_1(K_c|L^\perp,\Lambda|L^\perp) < \lambda=\lambda_1(K_c|\ov
L^\perp,\Lambda|\ov L^\perp)$. Hence, in order to determine
$\sup\{\lambda_1(K_c|L^\perp,\Lambda|L^\perp) : L \in
\grass(\Lambda,n-i)\}$ it suffices to consider only those lattice planes $L \in \grass(\Lambda,n-i)$ such that $\det(\Lambda\cap L)$ is upper bounded by the right hand side in \eqref{eq:obbound}.
We are now done since there are only finitely many lattice planes with bounded determinant (cf.~\cite[Prop.~1]{stuhler1976eine}).
\end{proof}

A fundamental result in the Geometry of Numbers is due to Banaszczyk~\cite{banaszczyk1996inequalities} who proved that for every origin-symmetric convex body~$K \in \Kn$, every lattice~$\Lambda \in \Lat^n$, and every $1 \leq i \leq n$, one has
\begin{align}
1 \leq \lambda_i(K,\Lambda) \cdot \lambda_{n-i+1}(\dual{K},\dual{\Lambda}) \leq \mathrm{c} n(1+\log(n)),\label{eq:banaszczyk}
\end{align}
for some absolute constant~$\mathrm{c} \geq 0$.
The lower bound is actually a classical result due to Mahler (cf.~\cite[Thm.~23.2]{gruber2007convex}).
Such relations between primal and polar bodies and lattices are usually called \emph{transference bounds}.

In the following we study similar relations between the packing minima of a convex body and its successive minima and covering minima.
We first prove Proposition~\ref{prop:succmin-bounds-intro}, illustrating that the sequence of packing minima interpolates between the successive minima of a convex body and the successive minima of its polar.


\begin{proof}[Proof of Proposition~\ref{prop:succmin-bounds-intro}]
Fix $i \in \{1,\ldots,n\}$.
By definition, $\rho_i(K,\Lambda)$ is the smallest number~$\rho \geq 0$ such that all projections of $\rho K_c$ along $(n-i)$-dimensional lattice planes $L$ of $\Lambda$ contain a non-trivial lattice point of the $i$-dimensional lattice $\Lambda|L^\perp$.
Now, let $\vv_1,\ldots,\vv_{n-i+1} \in \lambda_{n-i+1}(K_c,\Lambda) K_c \cap \Lambda$ be linearly independent lattice points.
Then, for every lattice plane $L \in \grass(\Lambda,n-i)$, we have
$\vv_j|L^\perp \in \Lambda|L^\perp \setminus \{\vnull\}$, for at least one index $j \in \{1,\ldots,n-i+1\}$.
Thus,
\begin{align*}
\rho_i(K,\Lambda) \leq \lambda_{n-i+1}(K_c,\Lambda).
\end{align*}
Moreover, $\rho_n(K,\Lambda)=\lambda_{1}(K_c,\Lambda)$ is clear from the definition.

On the other hand, for any  $L \in
\grass(\Lambda,n-i)$ we have by \eqref{eq:banaszczyk}
\begin{equation}
  \lambda_i(\dual{(K_c|L^\perp)},\dual{(\Lambda|L^\perp)}) \cdot 
  \lambda_1(K_c|L^\perp,\Lambda|L^\perp) \geq 1.
\label{eq:3}   
\end{equation}
Basic identities for polar bodies and lattices are $\dual{(K_c|L^\perp)} = \dual{K_c}\cap L^\perp$ and $\dual{(\Lambda|L^\perp)} = \dual{\Lambda}\cap
  L^\perp$ (see, e.g.,~\cite[Prop.~1.3.4]{martinet2003perfect}).
  Thus~\eqref{eq:3} reads
\begin{equation}
  \lambda_1(K_c|L^\perp,\Lambda|L^\perp)
  \geq \frac{1}{\lambda_i(\dual{K_c}\cap L^\perp,\dual{\Lambda}\cap
    L^\perp)}.
\label{eq:2}  
\end{equation} 
Hence, for $1\leq i \leq n$, 
\begin{equation*}
  \begin{split}
\rho_i(K,\Lambda) & = \max\left\{
   \lambda_1(K_c|L^\perp,\Lambda|L^\perp) : L \in
   \grass(\Lambda,n-i)\right\} \\ & \geq  
   \frac{1}{\min\{\lambda_i(\dual{K_c}\cap L',\dual{\Lambda}\cap L') : L' \in \grass(\dual{\Lambda},i)\}}  = \frac{1}{\lambda_i(\dual{K_c},\dual{\Lambda})}.
\end{split} 
\end{equation*} 
For $i=1$ we have equality in~\eqref{eq:3} and~\eqref{eq:2}, because $L^\perp$ is a line in this case.
Thus, $\rho_1(K,\Lambda)=1/\lambda_1(\dual{K_c},\dual{\Lambda})$ as
claimed (cf.~\eqref{eq:rho-1-n} and~\eqref{eq:mu1-lat-width}).
\end{proof}

We conjecture that there is actually a complete analog of the lower bound in~\eqref{eq:banaszczyk}, meaning that \emph{both} successive minima can be replaced by the corresponding packing minimum.

\begin{conj}
\label{conj:rho-transference}
Let $K \in \Kn$ and $\Lambda \in \Lat^n$.
Then, for $1 \leq j \leq n$, we have
\[
\rho_j(K,\Lambda) \cdot \rho_{n-j+1}(\dual{K_c},\dual{\Lambda}) \geq \frac12.
\]
\end{conj}

Compared to~\eqref{eq:banaszczyk}, the lower bound equals $1/2$ here, because with our definition of the difference body, we have $K_c = 2 K$, when $K$ is already origin-symmetric.
We provide some evidence for this inequality in Section~\ref{sect:reg-simplex-lattice}, in particular in Corollary~\ref{cor:Lambda-ab-bounds}.

Concerning relations between packing minima and covering minima, we establish the following linear inequalities.
In view of Proposition~\ref{prop:succmin-bounds-intro}, these bounds improve the corresponding relations between the covering minima and the successive minima (cf.~\cite[Thm.~23.4]{gruber2007convex} and~\cite[Lem.~(2.4) \& (2.5)]{kannanlovasz1988covering}).

\begin{proposition}
\label{prop:covmin-bounds}
Let $K \in \Kn$ and $\Lambda \in \Lat^n$.
\begin{enumerate}[label=\roman*)]
 \item For $1 \leq i \leq n-1$, we have
  \begin{align*}
      \mu_{i+1}(K,\Lambda)\leq \mu_i(K,\Lambda) + \rho_{i+1}(K,\Lambda).
  \end{align*}
 \item For $1 \leq i \leq n$, we have
  \begin{align*}
      \max_{1 \leq j \leq i}\rho_j(K,\Lambda) \leq \mu_i(K,\Lambda) \leq \sum_{j=1}^i \rho_j(K,\Lambda).
  \end{align*}
\end{enumerate}
\end{proposition}

\begin{proof}
i): The weaker bound $\mu_{i+1}(K,\Lambda) \leq \mu_i(K,\Lambda) + \lambda_{n-i}(K_c,\Lambda)$ has been proven by Kannan \& Lov\'{a}sz~\cite[Lem.~(2.5)]{kannanlovasz1988covering}.
However, their proof applies verbatim to the packing minima and establishes the claimed upper bound.

ii): The upper bound follows from the identity $\mu_1(K,\Lambda) =\rho_1(K,\Lambda)$ in~\eqref{eq:rho-1-n} and a successive application of the bound in~i).
For the lower bound, let $j \in \{1,\ldots,i\}$ and let $L \in \grass(\Lambda,n-j)$ be such that $\rho_j(K,\Lambda) = \lambda_1(K_c|L^\perp,\Lambda|L^\perp)$.
Applying the basic inequality $\lambda_1(M_c,\Lambda) \leq \mu_n(M,\Lambda)$ for $M \in \Kn$ and $\Lambda \in \Lat^n$ (cf.~\cite[\S 13.5]{gruberlekker1987geometry}) to $K_c|L^\perp$ and
$\Lambda|L^\perp$, and combining this with the definition and the monotonicity of the covering minima, we have 
\[
\rho_j(K,\Lambda) = \lambda_1(K_c|L^\perp,\Lambda|L^\perp) \leq \mu_j(K|L^\perp,\Lambda|L^\perp) \leq \mu_j(K,\Lambda) \leq \mu_i(K,\Lambda),
\]
as claimed.
\end{proof}

Based on the transference bound~\eqref{eq:banaszczyk}, we may complement the bounds in Proposition~\ref{prop:succmin-bounds-intro} by reverse inequalities.

\begin{proposition}
\label{prop:succmin-bounds-reverse}
Let $K \in \Kn$ and $\Lambda \in \Lat^n$.
Then, for $1\leq i\leq n$,
  \begin{equation*}
   \frac{\lambda_{n-i+1}(K_c,\Lambda)}{\mathrm{c} n(1+\log(n))} \leq \rho_i(K,\Lambda) \leq \frac{\mathrm{c} i (1+\log(i))}{\lambda_i(\dual{K_c},\dual{\Lambda})},
  \end{equation*}
where $\mathrm{c}$ is an absolute constant that is independent from~$K$ and~$\Lambda$.
\end{proposition} 

\begin{proof}
Let $L$ be an $(n-i)$-dimensional lattice plane of~$\Lambda$ attaining
$\rho_i(K,\Lambda)$.  
Then, $K_c|L^\perp$ is $i$-dimensional and by Banaszczyk's inequality~\eqref{eq:banaszczyk}, we have
\begin{align*}
\rho_i(K,\Lambda) &= \lambda_1(K_c|L^\perp,\Lambda|L^\perp) \leq \frac{\mathrm{c} i(1+\log(i))}{\lambda_i(\dual{(K_c|L^\perp)},\dual{(\Lambda|L^\perp)})} \\
&= \frac{\mathrm{c} i(1+\log(i))}{\lambda_i(\dual{K_c}\cap L,\dual{\Lambda}\cap L)} \leq \frac{\mathrm{c} i(1+\log(i))}{\lambda_i(\dual{K_c},\dual{\Lambda})}.
\end{align*}
The claimed lower bound follows by combining Banaszczyk's estimate~\eqref{eq:banaszczyk} with the lower bound in Proposition~\ref{prop:succmin-bounds-intro}.
In fact, we have
\[
\rho_i(K,\Lambda) \geq \frac{1}{\lambda_i(\dual{K_c},\dual{\Lambda})} \geq \frac{\lambda_{n-i+1}(K_c,\Lambda)}{\mathrm{c} n(1+\log(n))},
\]
for every $1 \leq i \leq n$.
\end{proof}

\section{Volume inequalities for packing minima}
\label{sect:volume-inequalities}

\noindent In this section, we are concerned with relating the volume of a convex body to its sequence of packing minima.
First, we strengthen some classical results from the Geometry of
Numbers, and second we discuss how they relate to
long-standing and famous conjectures of Mahler and Makai Jr.~in the field.

We begin with observing that due to $\lambda_1(K_c,\Lambda) = \rho_n(K,\Lambda)$, we may formulate Minkowski's 1st Theorem (cf.~\cite[Sect.~22 \& 23]{gruber2007convex}) as
\begin{align}
\frac{\vol(K)}{\det(\Lambda)} \leq \left(\frac{1}{\rho_n(K,\Lambda)}\right)^n = \left(\frac{1}{\lambda_1(K_c,\Lambda)}\right)^n,\label{eq:minkowskis-1st-rho-n}
\end{align}
where $K \in \Kn$ and $\Lambda \in \Lat^n$.

Just as for the successive minima, this is the only possible upper bound on the volume involving only one of the packing minima.
Indeed, by the degree of homogeneity of each of the $\rho_i(K,\Lambda)$, any such upper bound must be of the form $\vol(K) / \det(\Lambda) \leq \mathrm{c}_n \cdot \rho_i(K,\Lambda)^{-n}$, for some constant $\mathrm{c}_n \geq 0$.
However, setting $r_1 = \ldots = r_{n-1} = M^{-1}$ and $r_n = M^{n-1}$, for $M \geq 1$, the boxes $C(\vr)$ from the introduction have parameters $\vol(C(\vr)) = 2^n r_1\cdot\ldots\cdot r_n = 2^n$ and $\rho_i(C(\vr),\Z^n) = M/2$, for $1 \leq i \leq n-1$, preventing any such bound to hold in general, when $i \neq n$.


The situation is similar with respect to lower bounds on the
volume.
Suitable instances of $C(\vr)$ show again that the only possible inequality is
\begin{align}
\mathrm{c}_n \cdot \left(\frac{1}{\rho_1(K,\Lambda)}\right)^n \leq \frac{\vol(K)}{\det(\Lambda)}.\label{eq:makai-loose}
\end{align}
Due to the identities
\[
\frac{1}{\rho_1(K,\Lambda)} = \frac{1}{\mu_1(K,\Lambda)} = \lambda_1(\dual{K_c},\dual{\Lambda}) = 2 \cdot \rho_n(\dual{K_c},\dual{\Lambda}),
\]
this inequality can be understood in various different contexts.
The determination of the best possible constant $\mathrm{c}_n$ for the inequality~\eqref{eq:makai-loose} to hold in general is still an open problem.
In fact, in the language of so-called thinnest non-separable
arrangements, that is, within the context of $\mu_1(K,\Lambda)$, it is
an old conjecture of Makai Jr.~that the best bounds are as follows
(cf.~\cite{gonzalezmerinoschymura2017ondensities}):

\begin{conj}[{Makai Jr.~\cite{makai1978on}}]
\label{conj:makai}
Let $K \in \Kn$ and let $\Lambda \in \Lat^n$ be a lattice.
Then,
\begin{align}
\frac{1}{n!} \cdot \lambda_1(\dual{K_c},\dual{\Lambda})^n  &\leq    \frac{\vol(K)}{\det(\Lambda)}, \quad\text{ if }K=-K,\quad\text{ and}\label{eq:makai-symmetric}\\
\frac{n+1}{2^nn!} \cdot \lambda_1(\dual{K_c},\dual{\Lambda})^n  &\leq    \frac{\vol(K)}{\det(\Lambda)}, \quad\text{ for general } K.\label{eq:makai-general}
\end{align}
Moreover, the first inequality is attained for $\Lambda = \Z^n$ and the standard crosspolytope $K=\conv(\{\pm \ve_1,\ldots,\pm \ve_n\})$, and the second inequality is attained for $\Lambda = \Z^n$ and the simplex $K=\conv(\{\ve_1,\ldots,\ve_n,-(\ve_1+\ldots+\ve_n)\})$.
\end{conj}

We choose to formulate this conjecture with respect to the first successive minimum of the polar body because it illustrates that Makai Jr.'s conjecture is a polar analog of Minkowski's 1st Theorem~\eqref{eq:minkowskis-1st-rho-n}.
Informally it says that a small volume of~$K$ is a sufficient condition for the existence of non-trivial polar lattice points in the polar body~$\dual{K_c}$.

In combination with~\eqref{eq:minkowskis-1st-rho-n}, the bound~\eqref{eq:makai-symmetric} would follow from the notorious open conjecture of Mahler~\cite{mahler1939ein} about the minimal volume-product of a convex body.
It states that for every origin-symmetric $K \in \Kn$, one has
\begin{align}
\vol(K) \cdot \vol(\dual{K}) \geq \frac{4^n}{n!}.\label{eq:mahler}
\end{align}
Similarly, one would get~\eqref{eq:makai-general} from the non-symmetric variant
\[
\vol(K) \cdot \vol(\dual{K_c}) \geq \frac{n+1}{n!}, \quad\text{ for } K \in \Kn,
\]
of Mahler's conjecture formulated in~\cite[Conj.~3.17]{makaimartini2016density}.
The details of these connections have already been presented in~\cite{gonzalezmerinoschymura2017ondensities} and~\cite{henkxue2019onsuccessive}.

After this in-depth discussion of volume-inequalities for a single packing minimum, we now turn our attention to relating the volume functional to the whole sequence of packing minima.
We need to prepare our main observation with a simple but useful property, which is actually the packing minima analog of the corresponding inequality for covering minima (cf.~\cite[Lem.~4.2]{gonzalezmerinoschymura2017ondensities}).

\begin{lemma}
\label{lem:proj}
Let $K \in \Kn$, $\Lambda \in \Lat^n$, and let $1 \leq k \leq n-1$.
Then, for every lattice plane $L\in\grass(\Lambda,k)$ and every $1 \leq i \leq n-k$, we have
\[
\rho_i(K|L^\perp,\Lambda|L^\perp)\leq \rho_i(K,\Lambda).
\]
\end{lemma}

\begin{proof}
By definition of the packing minima and basic properties of orthogonal projections and complements, we see that for any $1\leq i\leq n-k$ holds
\begin{equation*}
  \begin{split}
    &\rho_i(K|L^\perp,\Lambda|L^\perp)\\ &=\max\left\{
   \lambda_1\big( (K_c|L^\perp) |H^\perp, (\Lambda|L^\perp)|H^\perp \big) : H \in
   \grass(\Lambda|L^\perp,n-k-i)\right\}  \\
 &=\max\left\{
   \lambda_1\big( (K_c|(L +H)^\perp, (\Lambda|(L+H)^\perp \big) : H \in
   \grass(\Lambda|L^\perp,n-k-i)\right\} \\
 & \leq \max\left\{
   \lambda_1\big( K_c|G^\perp, \Lambda|G^\perp \big) : G \in
   \grass(\Lambda|L^\perp,n-i)\right\} = \rho_i(K,\Lambda).\qedhere
\end{split}
\label{eq:5}
\end{equation*}  
\end{proof} 

\noindent We can now prove the variant of Minkowski's 2nd Theorem for packing minima, stated in Theorem~\ref{thm:minkowskisecondrho}.
In view of Proposition~\ref{prop:succmin-bounds-intro} and Proposition~\ref{prop:covmin-bounds}, the lower bound is stronger than the corresponding lower bound in~\eqref{eq:minkowskisecond} for successive minima, and the one in~\cite[Thm.~4.3]{gonzalezmerinoschymura2017ondensities} for covering minima.


\begin{proof}[Proof of Theorem~\ref{thm:minkowskisecondrho}]
As discussed, only the lower bound needs proof.
We use induction on $n$, and basically follow the ideas in~\cite[Thm.~4.3]{gonzalezmerinoschymura2017ondensities} for the corresponding inequality for covering minima.
We may assume $n \geq 2$.
Let $\vv\in\Lambda\setminus\{\vnull\}\cap \lambda_1(K_c,\Lambda)K_c$, which means that we find $\vv_1,\vv_2 \in K$ such that $\frac{1}{\lambda_1(K_c,\Lambda)}\vv = \vv_1-\vv_2$.
Writing $L_\vv = \lin\{\vv\}$, it follows, for instance by Steiner-Symmetrization (cf.~\cite[Sect.~9.1]{gruber2007convex}), that $\vol(K) \geq \frac{1}{n}\|\vv_1-\vv_2\|\cdot\vol(K|L_\vv^\perp)$.
Therefore,
  \begin{equation*}
    \begin{split} 
    \vol(K) & \geq \frac{1}{n}
    \|\vv_1-\vv_2\| \cdot \vol(K|L_\vv^\perp) = \frac{1}{n}\frac{1}{\lambda_1(K_c,\Lambda)}\|\vv\| \cdot \vol(K|L_\vv^\perp)\\
    & =\frac{1}{n}\frac{1}{\rho_n(K,\Lambda)}\det(\Lambda \cap L_\vv) \cdot  \vol(K|L_\vv^\perp).
\end{split} 
  \end{equation*}
Using that $\det(\Lambda \cap L_\vv)\det(\Lambda|L_\vv^\perp) = \det(\Lambda)$ (cf.~\cite[Prop.~1.9.7]{martinet2003perfect}), we have
   \begin{equation*}
    \begin{split} 
    \frac{\vol(K)}{\det(\Lambda)} & \geq \frac{1}{n}\frac{1}{\rho_n(K,\Lambda)} \frac{\vol(K|L_\vv^\perp)}{\det(\Lambda|L_\vv^\perp)},
\end{split} 
\end{equation*}
and thus arrive at a state where we may apply the induction hypothesis to $\vol(K|L_\vv^\perp)/\det(\Lambda|L_\vv^\perp)$ and obtain
   \begin{equation*}
    \begin{split} 
    \frac{\vol(K)}{\det(\Lambda)} & \geq \frac{1}{n!} \frac{1}{\rho_n(K,\Lambda)} \prod_{i=1}^{n-1} \frac{1}{\rho_i(K|L_\vv^\perp,\Lambda|L_\vv^\perp)}.
\end{split} 
\end{equation*}
The assertion now follows in view of Lemma~\ref{lem:proj}.
\end{proof}

As a consequence we get new evidence for Conjecture~\ref{conj:makai}: If $K \in \Kn$ and $\Lambda \in \Lat^n$ are such that $\rho_1(K,\Lambda) \geq \rho_i(K,\Lambda)$, for every $1 \leq i \leq n$, then the inequality~\eqref{eq:makai-symmetric} holds.
Additionally, the lower bound in Theorem~\ref{thm:minkowskisecondrho} is a
next step towards the strongest such estimate in terms of the product
of the functionals $\rho_i(K,\Lambda)^{-1}$,
$\lambda_i(K_c,\Lambda)^{-1}$ or
$\lambda_i(\dual{K_c},\dual{\Lambda})$, which is the following
strengthening of Makai Jr.'s Conjecture:
\begin{align}
\frac{1}{n!} \prod_{i=1}^n \lambda_i(\dual{K_c},\dual{\Lambda})  &\leq    \frac{\vol(K)}{\det(\Lambda)}, \quad\text{ if } K=-K,\quad\text{ and}\label{eq:lbM2-strongest-symmetric}\\
\frac{n+1}{2^nn!} \prod_{i=1}^n \lambda_i(\dual{K_c},\dual{\Lambda})  &\leq    \frac{\vol(K)}{\det(\Lambda)}, \quad\text{ for general } K \in \Kn.\label{eq:lbM2-strongest-general}
\end{align}
Just as in Conjecture~\ref{conj:makai}, but in contrast to the lower bound in~\eqref{eq:minkowskisecond}, the potential origin-symmetry of the involved convex body matters in these estimates.
This is due to the fact that polarization $K \mapsto \dual{K}$ and symmetrization $K \mapsto K_c$ as operations on convex bodies do not commute, so that $\lambda_i(\dual{K_c},\dual{\Lambda})$ and $\lambda_i((\dual{K})_c,\dual{\Lambda})$ are not comparable in general.

The bound~\eqref{eq:lbM2-strongest-symmetric} was conjectured in~\cite{henkxue2019onsuccessive}, and shown to hold for planar convex bodies in a stronger fashion.
Mahler~\cite{mahler1974polar} already proved a weaker inequality of this type.
Again, the validity of Mahler's conjecture~\eqref{eq:mahler} combined with Minkowski's 2nd Theorem~\eqref{eq:minkowskisecond} would imply~\eqref{eq:lbM2-strongest-symmetric}.
A weaker problem than~\eqref{eq:lbM2-strongest-symmetric} is to prove
\begin{align*}
\frac{1}{n!} \prod_{i=1}^n \rho_i(\dual{K_c},\dual{\Lambda})  &\leq    \frac{\vol(K)}{\det(\Lambda)}, \quad\text{ for origin-symmetric } K \in \Kn.
\end{align*}
In case that the transference bound in Conjecture~\ref{conj:rho-transference} holds, this inequality serves as an intermediate bound between the one in Theorem~\ref{thm:minkowskisecondrho} and~\eqref{eq:lbM2-strongest-symmetric}.
In view of the relations in Proposition~\ref{prop:covmin-bounds},
another relevant conjecture in this context is the \emph{covering
  product conjecture} posed
in~\cite{gonzalezmerinoschymura2017ondensities}.

\begin{conj}[{\cite[Conj.~4.8]{gonzalezmerinoschymura2017ondensities}}]
\label{conj:covering-product}
Let $K \in \Kn$ and $\Lambda \in \Lat^n$.
Then,
\[
\frac{n+1}{2^n} \prod_{i=1}^n \frac{1}{\mu_i(K,\Lambda)} \leq    \frac{\vol(K)}{\det(\Lambda)},
\]
and this is best possible.
\end{conj}

There is again no distinction between the case of origin-symmetric convex bodies and general ones.
In fact, there are supposed to be origin-symmetric convex bodies as well as certain simplices that attain equality in this estimate (see the discussion in~\cite{gonzalezmerinoschymura2017ondensities}).

\section{A discrete Minkowski-type inequality for packing minima}
\label{sect:main-result}

\noindent This section is devoted to proving our main result,
Theorem~\ref{thm:main}, which establishes an optimal upper bound on $\#(K\cap\Lambda)$ in
terms of the packing minima. First, however, we point
out that the weaker result, when the packing minima are replaced 
 by successive minima of the polar lattice, can be proved easily as
 claimed in the introduction. 
\begin{proposition} Let $K\in\Kn$ and $\Lambda\in\Lat^n$. Then
  \begin{equation*}
      \#(K\cap\Lambda) \leq \prod_{i=1}^n
      \left\lfloor \lambda_i(\dual{K_c},\dual{\Lambda})+1
      \right\rfloor
    \end{equation*}
    and the inequality is best possible. 
\label{prop:second-minkowski-polarlattice}
\end{proposition}
\begin{proof} For short we write $\lambda^\star_i$ instead of
  $\lambda_i(\dual{K_c},\dual{\Lambda})$, and let
  $\vv_1,\dots,\vv_n\in \dual{\Lambda}$ be linearly independent with
  $\vv_i\in \lambda^\star_i\cdot \dual{K_c}$, $1\leq i\leq n$. Then, by
  the definition of the polar body and the symmetry of $K_c$, we conclude 
  \begin{equation*}
    K-K\subseteq \{\vx\in\R^n : |\ip{\vv_i}{\vx}|\leq
    \lambda^\star_i, 1\leq i\leq n\}.
  \end{equation*}
Hence, if $\suk(M,\vu)=\max\{\ip{\vu}{\vx}: \vx\in M\}$, $\vu\in\R^n$,
denotes the support function of the convex body $M \in \Kn$, then
\begin{equation}
  \suk(K,\vv_i)+\suk(K,-\vv_i)=\suk(K-K,\vv_i)\leq\lambda^\star_i,\text{ for }  1\leq i\leq n.
  \label{eq:support-function}
\end{equation}
On the other hand, if we consider the parallelotope 
\begin{equation*} 
P= \{\vx\in\R^n :
-\suk(K,-\vv_i)\leq  \ip{\vv_i}{\vx}\leq \suk(K,\vv_i),\, 1\leq i\leq n\},
\end{equation*} 
we certainly have $K\subseteq P$  and on account of \eqref{eq:support-function} it suffices to observe that $\#(P\cap\Lambda)\leq
\prod_{i=1}^n  \lfloor \suk(K,\vv_i)+\suk(K,-\vv_i)+1\rfloor$. This,
however, follows from the fact that the map $\Lambda\to\Z^n$   given
by 
$\va\mapsto(\ip{\vv_1}{\va},\dots, \ip{\vv_n}{\va})^\trans$ is injective,
since  
$\vv_1,\dots,\vv_n\in\dual{\Lambda}$ are linearly independent.
\end{proof}   

For the proof of Theorem  \ref{thm:main} we need a few lemmas which we
state next. The first one extends the bound
~\eqref{eq:disc-lambda1-bound} to slices.

\begin{lemma}
\label{lem:slice_points} 
Let $K\in\Kn$, $\Lambda \in \Lat^n$, $L\in\grass(\Lambda,k)$,
$k\in\{1,\dots,n-1\}$, and $\vt\in\R^n$. Then
  \begin{equation*}
          \#(K\cap(\vt+L) \cap \Lambda) \leq \left\lfloor
            \frac{1}{\lambda_1(K_c\cap
              L,\Lambda\cap L)}+1\right\rfloor^k\leq \left\lfloor
            \frac{1}{\lambda_1(K_c,\Lambda)}+1\right\rfloor^k.
        \end{equation*}
\end{lemma} 

\begin{proof}
For short, let $\lambda_1=\lambda_1(K_c\cap
              L,\Lambda\cap L)$ and  $m=\lfloor
              1/\lambda_1+1\rfloor$. Assume that
              $\#(K\cap(\vt+L)\cap \Lambda)> m^k$. Then there exist two
              different lattice points $\vv,\vw\in K\cap(\vt+L)\cap \Lambda$ belonging
              to the same residue class modulo $m$, which means that
              \begin{equation*}          
                \frac{1}{m}(\vv-\vw)
                \in \Lambda\setminus\{\vnull\}\cap\left[\frac{1}{m} (K_c\cap L)\right].
              \end{equation*}
              Since $1/m<\lambda_1$, this contradicts the definition
              of $\lambda_1$, and we have shown the first inequality. Obviously, $\lambda_1=\lambda_1(K_c\cap
              L,\Lambda\cap L)\geq \lambda_1(K_c,\Lambda)$ which gives
              the second inequality.
\end{proof} 

The next lemma establishes relationships between optimal lattice
planes for which  $\rho_k(K,\Lambda)$ is achieved  and lattice points
corresponding to the first successive minimum, i.e., $\lambda_1(K_c,\Lambda)$.

\begin{lemma}
\label{lem:firstvector}   
Let $K\in\Kn$, $\Lambda \in \Lat^n$ and $k\in\{1,\dots,n-1\}$.
Furthermore, let $L \in \grass(\Lambda,n-k)$ be such that
$\rho_k(K,\Lambda)=\lambda_1(K_c|L^\perp,\Lambda|L^\perp)$, and let $\vv \in (\lambda_1(K_c,\Lambda)\,K_c) \cap \Lambda\setminus\{\vnull\}$.
  \begin{enumerate}
  \item If $\rho_k(K,\Lambda)>\lambda_1(K_c,\Lambda)$, then $\vv\in
    L$.
  \item If $\vv\in L$, then $\rho_k(K,\Lambda)=\rho_k(K|\vv^\perp,\Lambda|\vv^\perp)$.
  \end{enumerate}
\end{lemma}

\begin{proof}
For i) we suppose $\vv\notin L$.
Then,
\begin{equation*} 
\vv|L^\perp \in \lambda_1(K_c,\Lambda)\left(K_c|L^\perp\right)
\cap\left(\Lambda|L^\perp \setminus\{\vnull\}\right)
\end{equation*}
and thus  
\begin{equation*} 
\rho_k(K,\Lambda)=\lambda_1(K_c|L^\perp,\Lambda|L^\perp)\leq
\lambda_1(K_c,\Lambda).
\end{equation*}
For ii) we write $L = \lin\{\vv\} \oplus L^\prime$  with
$L^\prime=L\cap \vv^\perp\in \grass(\Lambda|\vv^\perp,n-k-1)$. Then 
\begin{equation*}
  \begin{split}
    \rho_k(K,\Lambda) & =
    \lambda_1(K_c|L^\perp,\Lambda|L^\perp)\\ & =\lambda_1((K_c|\vv^\perp)|(L^\prime)^\perp,(\Lambda|\vv^\perp)|(L^\prime)^\perp)
    \leq \rho_k(K|\vv^\perp,\Lambda|\vv^\perp). 
\end{split} 
\end{equation*}
Together with Lemma \ref{lem:proj} we get ii). 
\end{proof} 

We have already seen in the introduction that the packing minima do not necessarily form a decreasing sequence.
However, the next observation will help us to identify useful properties of a maximal strictly decreasing subsequence, which will be essential for our proof of Theorem~\ref{thm:main}.

\begin{lemma}
\label{lem:main}
Let  $K\in\Kn$, $\Lambda \in \Lat^n$, and let 
\begin{equation*} 
S = \left\{j\in\{1,\dots,n-1\} :
  \rho_j(K,\Lambda)>\rho_n(K,\Lambda)\right\}\ne\emptyset. 
\end{equation*} 
Let $k^*=\max\{i: i\in S\}$  be the maximal element in $S$. 
Then, there exists an $L \in \grass(\Lambda,n-k^*)$ such that for all
$j\in S$ 
  \begin{equation*}
           \rho_j(K,\Lambda)=\rho_j(K|L^\perp,\Lambda|L^\perp).
         \end{equation*}
\end{lemma} 

\begin{proof}
Let $\vv \in (\lambda_1(K_c,\Lambda) K_c) \cap \Lambda\setminus\{\vnull\}$.
According to Lemma~\ref{lem:proj} and Lemma~\ref{lem:firstvector}, we
have for $j\in\{1,\dots,n-1\}$ 
  \begin{equation}
    \begin{split} 
     \rho_j(K|\vv^\perp,\Lambda|\vv^\perp) & =
     \rho_j(K,\Lambda),\textrm{ if } j\in S,
    \\ \rho_j(K|\vv^\perp,\Lambda|\vv^\perp) &\leq \rho_j(K,\Lambda)\leq
    \rho_n(K,\Lambda),\textrm{ if } j\notin  S.
  \end{split}
\label{eq:onevector}  
  \end{equation}
  If $k^*=n-1$, we are done with $L = \lin\{\vv\}$.
  Thus, we assume that $k^*<n-1$.
  By~\eqref{eq:onevector} we have, with $K^\prime=K|\vv^\perp$ and $\Lambda^\prime=\Lambda|\vv^\perp$, that for every $j\in S$ 
\begin{equation*}
  \begin{split} 
  \rho_j(K^\prime,\Lambda^\prime) &=
  \rho_j(K,\Lambda) >\rho_n(K,\Lambda) \\ &\geq \rho_{n-1}(K,\Lambda)\geq
  \rho_{n-1}(K^\prime,\Lambda^\prime) 
   = \lambda_1(K^\prime_c,\Lambda^\prime).
 \end{split}  
\end{equation*}
We can therefore apply the same argumentation to $K^\prime$ and
$\Lambda^\prime$.
More precisely, with  $\vv^\prime \in
(\lambda_1(K^\prime_c,\Lambda^\prime) K^\prime_c) \cap
\Lambda^\prime\setminus\{\vnull\}$, we again get by
Lemma~\ref{lem:proj} and Lemma~\ref{lem:firstvector} that for
$j\in\{1,\dots,n-2\}$, 
 \begin{equation*}
    \begin{split} 
     \rho_j(K^\prime|(\vv^\prime)^\perp,\Lambda^\prime|(\vv^\prime)^\perp) & = \rho_j(K^\prime,\Lambda^\prime)=\rho_j(K,\Lambda),\textrm{ if } j\in S,
    \\
    \rho_j(K^\prime|(\vv^\prime)^\perp,\Lambda^\prime|(\vv^\prime)^\perp)
    &\leq
    \rho_j(K^\prime,\Lambda^\prime)\leq
    \rho_j(K,\Lambda)\leq
    \rho_n(K,\Lambda),\textrm{ if } j\notin S.
  \end{split}
  \end{equation*}
  If $k^*=n-2$, we are done with $L = \lin\{\vv, \vv^\prime\}$.
  Otherwise, we set $K''=K^\prime|(\vv^\prime)^\perp$ and $\Lambda''=\Lambda^\prime|(\vv^\prime)^\perp$.
  Then, for every $j\in S$
\begin{equation*}
  \begin{split} 
  \rho_j(K'',\Lambda'') &=
  \rho_j(K,\Lambda) >\rho_n(K,\Lambda) \\ &\geq \rho_{n-2}(K,\Lambda)\geq
  \rho_{n-2}(K'',\Lambda'') 
   = \lambda_1(K''_c,\Lambda'')
 \end{split}  
\end{equation*}  
and we may proceed iteratively.
\end{proof} 

\begin{rem} In view of Lemma \ref{lem:firstvector} i), 
the proof of Lemma~\ref{lem:main} also gives the following:
For $j\in S$, let $M_{n-j} \in \grass(\Lambda,n-j)$ be such that $\rho_j(K,\Lambda) = \lambda_1(K_c|M_{n-j}^\perp,\Lambda|M_{n-j}^\perp)$.
Then, the subspace $L$, iteratively constructed in the proof of the 
lemma,   satisfies  $L \subseteq M_{n-j}$. 
\end{rem}

We are now prepared to prove our main inequality regarding an upper
bound on the lattice
point enumerator in terms of the packing minima. Our main
result, i.e., Theorem \ref{thm:main}, will be an immediate consequence
of it.

\begin{theo}
\label{theo:main2} 
Let $K \in \Kn$ and $\Lambda \in \Lat^n$.
Further, let $0 = j_0 < j_1 < \ldots < j_m < j_{m+1} = n$ with $m \in
\Z_{\geq 0}$ be such that
  \begin{enumerate} 
  \item  $\rho_{j_1}(K,\Lambda) > \rho_{j_2}(K,\Lambda) > \ldots > \rho_{j_m}(K,\Lambda) > \rho_n(K,\Lambda)$, and
   \item $\rho_{j}(K,\Lambda) \leq \rho_{j_i}(K,\Lambda)$, for
     $j \in \{j_{i-1}+1,\ldots,j_i\}$ and $i \in \{1,\ldots,m+1\}$.
   \end{enumerate}
 In words, the $\rho_{j_i}(K,\Lambda)$, $i=1,\dots, m+1$ form a
 maximal strictly decreasing subsequence of $\rho_1(K,\Lambda),\ldots,\rho_n(K,\Lambda)$.
 Then, 
 \begin{equation*}
    \#(K\cap\Lambda)\leq \prod_{i=1}^{m+1} \left\lfloor\frac{1}{\rho_{j_i}(K,\Lambda)}+1\right\rfloor^{j_i-j_{i-1}}. 
  \end{equation*}
\end{theo} 

\begin{proof}
If $m=0$, the statement follows directly from~\eqref{eq:disc-lambda1-bound}.
So let $m>0$.
Applying Lemma~\ref{lem:main} with $k^*=j_m$, gives an $L\in\grass(\Lambda,n-k^*)$ such that for every $i \in \{1,\ldots,m\}$
 \begin{equation}
    \rho_{j_i}(K,\Lambda)=\rho_{j_i}(K|L^\perp,\Lambda|L^\perp).
   \label{eq:7} 
 \end{equation}
Let $\vt \in \R^n$ be such that $\#(K \cap (\vt + L) \cap \Lambda)$ is maximal.
Then, an application of Lemma~\ref{lem:slice_points} gives us
 \begin{equation*}
   \begin{split} 
 \#(K\cap\Lambda) &\leq
 \#(K|L^\perp \cap \Lambda|L^\perp) \cdot \#(K \cap (\vt + L) \cap \Lambda)
 \\&\leq
 \#(K|L^\perp \cap \Lambda|L^\perp) \cdot \left\lfloor\frac{1}{\lambda_1(K_c,\Lambda)}+1\right\rfloor^{n-k^*}
 \\& =
 \#(K|L^\perp \cap \Lambda|L^\perp)\cdot \left\lfloor\frac{1}{\rho_{j_{m+1}}(K,\Lambda)}+1\right\rfloor^{j_{m+1}-j_m}.
\end{split} 
\end{equation*}
In view of Lemma \ref{lem:proj} and \eqref{eq:7}, the
sequence $\rho_{j_i}(K|L^\perp,\Lambda|L^\perp) =
\rho_{j_i}(K,\Lambda)$, with $i=1,\dots,m$, forms a maximal strictly decreasing
subsequence of $\rho_1(K|L^\perp,\Lambda|L^\perp),\ldots,\rho_{k^*}(K|L^\perp,\Lambda|L^\perp)$.
Hence, the assertion follows by induction on~$m$.
\end{proof}

Now the proof of Theorem \ref{thm:main} follows right from the
theorem above.
\begin{proof}[Proof of Theorem \ref{thm:main}] In view of property ii) in
  Theorem ~\ref{theo:main2}, it implies  
  \begin{equation*}
    \begin{split} 
    \#(K\cap\Lambda) \leq \prod_{i=1}^{m+1}
    \left\lfloor\frac{1}{\rho_{j_i}(K,\Lambda)}+1\right\rfloor^{j_i-j_{i-1}}
    \leq \prod_{i=1}^{n}
    \left\lfloor\frac{1}{\rho_{i}(K,\Lambda)}+1\right\rfloor.\qedhere
    \end{split}
  \end{equation*}
\end{proof}

\noindent We finish our discussion of the relationship between the number of lattice points and the packing minima with a proof of the lower bound on $\#(K \cap \Lambda)$ for planar convex bodies, stated in the introduction as Theorem~\ref{thm:discrete-lower-bound-plane}.

\begin{proof}[Proof of Theorem~\ref{thm:discrete-lower-bound-plane}]
We only discuss the details of the bound for arbitrary, that is, not necessarily origin-symmetric, convex bodies $K \in \mathcal{K}^2$.
If $K$ is origin-symmetric then a suitable adjustment of the argument below leads to the slightly improved bound stated in the theorem.

It suffices to prove the bound for $\Lambda = \Z^2$, and because $\rho_1(K,\Lambda) = \lambda_1(\dual{K_c},\dual{\Lambda})^{-1}$ and $\rho_2(K,\Lambda) = \lambda_1(K_c,\Lambda)$, we prove it in the formulation
\begin{align}
\#(K \cap \Z^2) &\geq \frac12 \left\lfloor \lambda_1(\dual{K_c},\Z^2)-2 \right\rfloor \left( \frac{1}{\lambda_1(K_c,\Z^2)} - 2 \right)-2.\label{eq:low-bound-rho-plane-gen}
\end{align}
First, write $\lambda_1 = \lambda_1(K_c,\Z^2)$, $\lambda_1^\star = \lambda_1(\dual{K_c},\Z^2)$, and let $\vv \in \lambda_1 \cdot K_c \cap \Z^2 \setminus \{\vnull\}$.
We apply a suitable unimodular transformation to assume that $\vv = \ve_1$.
Moreover, let $\vv_1,\vv_2 \in K$ be such that $\vv =
\lambda_1(\vv_1-\vv_2)$, and we translate~$K$ by the lattice vector $-\lfloor\ip{\vv_1}{\ve_2}\rfloor\ve_2$, so that (after the translation) the second coordinate of the points $\vv_1$ and $\vv_2$ lies in $[0,1)$.
We now fix the parameters
\begin{align*}
\gamma_2 &= \suk(K,\ve_2) = \max\{\ip{\vx}{\ve_2} : \vx \in K\},\\
\gamma_1 &= \suk(K,-\ve_2) = \max\{-\ip{\vx}{\ve_2} : \vx \in K\},\text{ and}\\
w &= \ip{\vv_1}{\ve_2} = \ip{\vv_2}{\ve_2} \in [0,1).
\end{align*}
Clearly, $\gamma_2 - w \geq 0$ and $\gamma_1 + w \geq 0$, and at most one of these inequalities can be tight as $K$ is two-dimensional.
Since $\lambda_1(\dual{K_c},\Z^2)$ is the lattice-width of~$K$ and from the definition of~$\vv$, we also have
\begin{align}
\gamma_1 + \gamma_2 \geq \lambda_1^\star \qquad \text{and} \qquad \|\vv_1-\vv_2\| = \frac{1}{\lambda_1}\|\vv\| = \frac{1}{\lambda_1}.\label{eq:eta-lambda-1-dual-gen}
\end{align}
Further, let $\vs,\vt \in K$ be such that $\gamma_2 = \ip{\vs}{\ve_2}$ and $\gamma_1 = \ip{\vt}{-\ve_2}$.
By convexity of~$K$, we then have $K \supseteq \conv\{\vv_1,\vv_2,\vs,\vt\} =: Q$, so that the number of lattice points in~$K$ can be estimated by
\begin{align}
\#(K \cap \Z^2) &\geq \#(Q \cap \Z^2) = \sum_{i=-\lfloor \gamma_1 \rfloor}^{\lfloor \gamma_2 \rfloor} \#(Q \cap (\R \ve_1 + i \ve_2) \cap \Z^2)\nonumber\\
&\geq \sum_{i=1}^{\lfloor \gamma_2 \rfloor} \left\lfloor \vol_1(Q \cap (\R \ve_1 + i \ve_2)) \right\rfloor + \sum_{j=0}^{\lfloor \gamma_1 \rfloor} \left\lfloor \vol_1(Q \cap (\R \ve_1 - j \ve_2)) \right\rfloor.\label{eq:s-t-estimate}
\end{align}
Here, $\vol_1(\cdot)$ denotes the one-dimensional volume.
Note that if either $\gamma_2 < 1$ or $\gamma_1 < 0$, then the corresponding sum would be empty and there is nothing to estimate.
So we assume in the sequel that both $\gamma_2 \geq 1$ and $\gamma_1 \geq 0$ hold.

We now use the intercept theorem for rays emanating from~$\vs$, respectively~$\vt$, and passing through $\vv_1$ and $\vv_2$, in order to obtain
\[
\vol_1(Q \cap (\R \ve_1 + i \ve_2)) = \frac{\gamma_2}{\gamma_2-w} \frac{1}{\lambda_1}\left(1-\frac{i}{\gamma_2}\right) = \frac{1}{\lambda_1}\frac{\gamma_2-i}{\gamma_2-w}, \text{ for } 1 \leq i \leq \lfloor \gamma_2 \rfloor,
\]
and respectively
\[
\vol_1(Q \cap (\R \ve_1 - j \ve_2)) = \frac{\gamma_1}{\gamma_1+w} \frac{1}{\lambda_1}\left(1-\frac{j}{\gamma_1}\right) = \frac{1}{\lambda_1}\frac{\gamma_1-j}{\gamma_1+w}, \text{ for } 0 \leq j \leq \lfloor \gamma_1 \rfloor.
\]
Taking into account that $0 \leq w \leq 1$, we continue the estimate~\eqref{eq:s-t-estimate} by
\begin{align*}
\#(K \cap \Z^2)\phantom{\hspace{-30pt}}&\\
&\geq \sum_{i=1}^{\lfloor \gamma_2 \rfloor} \left\lfloor \frac{1}{\lambda_1}\frac{\gamma_2-i}{\gamma_2-w} \right\rfloor + \sum_{j=0}^{\lfloor \gamma_1 \rfloor} \left\lfloor \frac{1}{\lambda_1}\frac{\gamma_1-j}{\gamma_1+w} \right\rfloor\\
&\geq \sum_{i=1}^{\lfloor \gamma_2 \rfloor} \left( \frac{1}{\lambda_1}\frac{\gamma_2-i}{\gamma_2-w} - 1 \right) + \sum_{j=0}^{\lfloor \gamma_1 \rfloor} \left( \frac{1}{\lambda_1}\frac{\gamma_1-j}{\gamma_1+w} - 1 \right)\\
&= \frac{\lfloor \gamma_2 \rfloor(2 \gamma_2 - \lfloor \gamma_2 \rfloor - 1)}{2 \lambda_1 (\gamma_2-w)} - \lfloor \gamma_2 \rfloor + \frac{\lfloor \gamma_1 +1\rfloor(2\gamma_1 - \lfloor \gamma_1 \rfloor)}{2 \lambda_1 (\gamma_1+w)} - \lfloor \gamma_1 +1\rfloor\\
&\geq \frac{\lfloor \gamma_2 \rfloor(\gamma_2-1)}{2\lambda_1(\gamma_2-w)} - \lfloor \gamma_2 \rfloor + \frac{\lfloor \gamma_1+1 \rfloor\gamma_1}{2\lambda_1(\gamma_1+w)} - \lfloor \gamma_1+1 \rfloor\\
&\geq \frac{\lfloor \gamma_2 \rfloor(\gamma_2-1)}{2\lambda_1\gamma_2} - \lfloor \gamma_2 \rfloor + \frac{\lfloor \gamma_1+1 \rfloor\gamma_1}{2\lambda_1(\gamma_1+1)} - \lfloor \gamma_1+1 \rfloor\\
&= \left(\frac{1}{2\lambda_1}-1\right)\left(\lfloor\gamma_2\rfloor+\lfloor\gamma_1\rfloor+1\right) - \frac{1}{2\lambda_1}\left(\frac{\lfloor\gamma_2\rfloor}{\gamma_2} + \frac{\lfloor\gamma_1+1\rfloor}{\gamma_1+1}\right)\\
&\geq \left(\frac{1}{2\lambda_1}-1\right)\lfloor\gamma_2+\gamma_1\rfloor - \frac{1}{\lambda_1} \ \geq \ \left(\frac{1}{2\lambda_1}-1\right)\lfloor\lambda_1^\star\rfloor - \frac{1}{\lambda_1}\\
&= \frac12\left(\frac{1}{\lambda_1}-2\right)\lfloor \lambda_1^\star-2\rfloor - 2.
\end{align*}
In the penultimate line we also used the assumption $\lambda_1 = \rho_2(K,\Z^2) \leq\frac12$.
\end{proof}

\section{Lattices generated by an equiangular basis}
\label{sect:reg-simplex-lattice}

\noindent In this section, we exhibit examples illustrating the following claims:

\begin{enumerate}[label=\alph*)]
 \item In contrast to the boxes $C(\vr)$ in~\eqref{eq:ex-boxes}, the sequence of packing minima is not necessarily decreasing.
 \item The computation of $\rho_j(K,\Lambda)$ is challenging, even for very structured examples of~$K$ and~$\Lambda$, and particular cases such as $j=n-1$.
 \item The packing minima should obey a strong transference theorem that refines the lower bound in~\eqref{eq:banaszczyk} (see Conjecture~\ref{conj:rho-transference}).
\end{enumerate}
We are interested in the packing minima of the Euclidean unit ball~$B_n = \{\vx \in \R^n : \|\vx\|\leq 1\}$ with respect to lattices that are generated by vectors of the same length, and such that any two of them span the same angle.
To this end, given $a > 0$ and $b \in \R$, we say that a basis $\{\vv_1,\ldots,\vv_n\}$ of $\R^n$ is \emph{$(a,b)$-equiangular} if
\begin{align}
\ip{\vv_k}{\vv_\ell} = \begin{cases} a & \text{for } k=\ell,\\ b & \text{for }k \neq \ell.\end{cases}\label{eq:lambda-ab-basis}
\end{align}
We denote the lattice that is generated by the vectors $\vv_1,\ldots,\vv_n$ by~$\Lambda_{a,b}$, and note that it is determined up to orthogonal transformations.
Since~$B_n$ is invariant under orthogonal transformations, the particular choice of~$\Lambda_{a,b}$ does not matter.
The lattices $\Lambda_{a,b}$ interpolate between two special lattices:
\begin{itemize}
 \item $\Lambda_{1,0}$ is orthogonally equivalent to the standard lattice $\Z^n$, and
 \item $\Lambda_{1,\frac12}$ is generated by the edge-directions of a regular simplex of edge-length one.
\end{itemize}
Although the most natural example to consider is the \emph{regular simplex lattice}~$\Lambda_{1,\frac12}$, we choose to investigate the parameterized lattices $\Lambda_{a,b}$ because this allows to deal more uniformly with projections and polarity.
Before we can go into the details, we need two preparatory lemmas.
The first one just determines the polar lattice of~$\Lambda_{a,b}$.

\begin{lemma}
\label{lem:props_lambda-ab}
Let $G_{a,b} = (\ip{\vv_k}{\vv_\ell})_{k,\ell}$ be the
Gram  matrix of an $(a,b)$-equi\-angular basis $\{\vv_1,\ldots,\vv_n\}$ and let $\Lambda_{a,b}$ be the corresponding lattice.
Then, $G_{a,b}^{-1}$ is the Gram matrix of the dual basis
$\{\bar\vv_1,\ldots,\bar\vv_n\}$ of~$\dual{\Lambda_{a,b}}$ and
it holds
\[
G_{a,b}^{-1} = \frac{1}{a a' + (n-1) b b'} \, G_{a',b'} = G_{\bar a,\bar b},
\]
where $a' = a + (n-2)b$, $b' = -b$, $\bar a = \frac{a'}{a a' + (n-1) b b'}$ and $\bar b = \frac{b'}{a a' + (n-1) b b'}$.

In particular, we have $\dual{\Lambda_{a,b}} = \Lambda_{\bar a,\bar b}$.
\end{lemma}

\begin{proof}
Let us write $V = (\vv_1,\ldots,\vv_n) \in \R^{n \times n}$, so that $\Lambda_{a,b} = V \Z^n$.
Then, by the definition of the polar lattice, we have that $\dual{\Lambda_{a,b}} = V^{-\intercal} \Z^n$, and the columns of $V^{-\intercal}$ are the dual basis vectors~$\bar\vv_i$, $1 \leq i \leq n$.
The corresponding Gram matrix is thus given as
\[
(V^{-\intercal})^\intercal V^{-\intercal} = (V^\intercal V)^{-1} = G_{a,b}^{-1}.
\]
In order to see that $G_{\bar a,\bar b} = G_{a,b}^{-1}$, we first note that due to the symmetry of these matrices, we only need to check that the upper left entry of $G_{\bar a,\bar b} \cdot G_{a,b}$ equals $1$ and that the second entry in the first row of this matrix vanishes.
Indeed, the upper left entry equals
\[
\ip{(\bar a,\bar b,\ldots,\bar b)}{(a,b,\ldots,b)} = \frac{aa'}{aa'+(n-1)bb'} + \frac{(n-1)bb'}{aa'+(n-1)bb'} = 1,
\]
whereas the second entry in the first row equals
\[
\ip{(\bar a,\bar b,\ldots,\bar b)}{(b,a,b,\ldots,b)} = \frac{a'b+ab'+(n-2)bb'}{aa'+(n-1)bb'} = 0,
\]
as claimed.
\end{proof}

The second lemma is just a technical auxiliary bound needed below.

\begin{lemma}
\label{lem:integer-beta-inequality}
Let $(\beta_1,\ldots,\beta_n) \in \Z^n \setminus \{\vnull\}$.
Then,
\[
2 \cdot \sum_{1 \leq k < \ell \leq n} \beta_k \beta_\ell \leq n \cdot \Big(\sum_{k=1}^n \beta_k^2 - 1\Big).
\]
\end{lemma}

\begin{proof}
First of all elementary algebraic manipulations show that the claimed inequality is equivalent to
\[
\sum_{1 \leq k < \ell \leq n} (\beta_k - \beta_\ell)^2 + \sum_{k=1}^n \beta_k^2 \geq n.
\]
Now, after suitably relabeling the indices we may assume that $\beta_1,\ldots,\beta_j \neq 0$ and $\beta_{j+1}=\ldots=\beta_n=0$, for some $j \in \{1,\ldots,n\}$.
Then,
\[
\sum_{1 \leq k < \ell \leq n} (\beta_k - \beta_\ell)^2 + \sum_{k=1}^n \beta_k^2 \geq j(n-j) + j = n + (j-1)(n-j) \geq n,
\]
as desired.
\end{proof}

Let us now determine the successive minima of $B_n$ with respect to the lattice~$\Lambda_{a,b}$, and with respect to projections along lattice planes that are generated by a subset of the basis vectors $\vv_i$ in~\eqref{eq:lambda-ab-basis}.

\begin{proposition}
\label{prop:ex-ball-symmetric-lattice}
Let $a > 0$ and $b \in \R$ be such that
\[
a \geq 2 b \geq 0 \qquad \text{or} \qquad a \geq -n b \geq 0,
\]
and let $\{\vv_1,\ldots,\vv_n\}$ be an $(a,b)$-equiangular basis.
\begin{enumerate}[label=\roman*)]
 \item For every $i \in \{1,\ldots,n\}$, we have
 \[
 \lambda_i(B_n,\Lambda_{a,b}) = \sqrt{a}.
 \]
 \item\label{item:Lambda-ab-sucmins} For $j \in \{1,\ldots,n-1\}$, let $L_j = \lin\{\vv_1,\ldots,\vv_j\} \in \grass(\Lambda_{a,b},j)$.
 Then,
 \[
 \lambda_i(B_n|L_j^\perp,\Lambda_{a,b}|L_j^\perp) = \sqrt{\frac{(a-b)(a + b j)}{a + b(j-1)}},
 \]
 for every $i \in \{1,\ldots,n-j\}$.
\end{enumerate}
\end{proposition}

\begin{proof}
For i), we first observe that the upper bound $\lambda_i(B_n,\Lambda_{a,b}) \leq \sqrt{a}$ is immediate since the $\vv_i$ are linearly independent lattice vectors of length $\sqrt{a}$.
For the lower bound, due to the monotonicity of the successive minima, it suffices to show that $\lambda_1(B_n,\Lambda_{a,b}) \geq \sqrt{a}$, that is, $\|\vw\|^2 \geq a$, for every $\vw \in \Lambda_{a,b} \setminus \{\vnull\}$.
Any such vector can be written as $\vw = \beta_1 \vv_1 + \ldots + \beta_n \vv_n$, for some $\vbeta = (\beta_1,\ldots,\beta_n) \in \Z^n \setminus \{\vnull\}$.
In view of~\eqref{eq:lambda-ab-basis}, we thus claim that
\begin{align}
\|\vw\|^2 &= a \cdot \left(\sum_{k=1}^n \beta_k^2\right)  + 2b \cdot \left(\sum_{1 \leq k < \ell \leq n}\beta_k \beta_\ell\right) \geq a.\label{eq:w-length-bound}
\end{align}
\underline{Case 1:} $a \geq 2b \geq 0$.

If $\sum_{k < \ell}\beta_k \beta_\ell \geq 0$, then since $b \geq 0$ and $\sum_{k=1}^n \beta_k^2 \geq 1$ by the assumption that $\vbeta \neq \vnull$, the bound~\eqref{eq:w-length-bound} clearly holds.
If $\sum_{k < \ell} \beta_k \beta_\ell \leq -1$, then 
\begin{align*}
&\phantom{=\ }a \cdot \left(\sum_{k=1}^n \beta_k^2\right) + 2b \cdot \left(\sum_{1 \leq k < \ell \leq n}\beta_k \beta_\ell\right) \\
&=a \cdot \left(\sum_{k=1}^n \beta_k\right)^2 + 2(b-a) \cdot \left(\sum_{1 \leq k < \ell \leq n}\beta_k \beta_\ell\right) \\
&\geq a \cdot \left(\sum_{k=1}^n \beta_k\right)^2 - 2(b-a) \geq 2(a-b),
\end{align*}
which is at least~$a$, in view of $a \geq 2b$.

\noindent\underline{Case 2:} $a \geq -nb \geq 0$.

In order to prove~\eqref{eq:w-length-bound}, that is,
\[
a \cdot \left(\sum_{k=1}^n \beta_k^2 - 1\right)  + 2b \cdot \left(\sum_{1 \leq k < \ell \leq n}\beta_k \beta_\ell\right) \geq 0,
\]
we may use $a \geq -nb$ and $b \leq 0$, and see that it suffices to show
\[
2 \cdot \left(\sum_{1 \leq k < \ell \leq n} \beta_k \beta_\ell\right) \leq n \cdot \left(\sum_{k=1}^n \beta_k^2 - 1\right).
\]
This inequality is proven in Lemma~\ref{lem:integer-beta-inequality}.

For ii), we show that the projected lattice $\Lambda_{a,b}|L_j^\perp$ is an $(n-j)$-dimensional lattice of the form $\Lambda_{\tilde a,\tilde b} \in \Lat^{n-j}$, where
\[
\tilde a = \frac{(a-b)(a + b j)}{a + b(j-1)} \quad \text{and} \quad \tilde b = \frac{(a-b)b}{a + b(j-1)}.
\]
Moreover, elementary algebraic manipulations show that if $a \geq 2b \geq 0$, then $\tilde a \geq 2 \tilde b \geq 0$, and likewise, if $a \geq -nb \geq 0$, then $\tilde a \geq -(n-j) \tilde b \geq 0$.
Since $B_n|L_j^\perp$ is an $(n-j)$-dimensional Euclidean unit ball, we can therefore apply the claim in part~i) and get that indeed $\lambda_i(B_n|L_j^\perp,\Lambda_{a,b}|L_j^\perp) = \sqrt{\tilde a}$, for every $i \in \{1,\ldots,j\}$.

In order to see that $\Lambda_{a,b}|L_j^\perp = \Lambda_{\tilde a,\tilde b}$, we first note that, for $j+1 \leq k \leq n$,
\[
\vw_k := \vv_k | L_j^\perp = \vv_k - \eta_j (\vv_1+\ldots+\vv_j),
\]
where $\eta_j = \frac{b}{a+b(j-1)}$.
Indeed, for every $1 \leq i \leq j$, we have
\[
\ip{\vw_k}{\vv_i} = \ip{\vv_k}{\vv_i} - \eta_j (\ip{\vv_1}{\vv_i} + \ldots + \ip{\vv_j}{\vv_i}) = b - \eta_j (a+b(j-1)) = 0.
\]
Moreover, since $\{\vv_1,\ldots,\vv_n\}$ is a basis of $\Lambda_{a,b}$, the set $\{\vw_{j+1},\ldots,\vw_n\}$ is a basis of $\Lambda_{a,b} | L_j^\perp$.
The squared length of the vectors~$\vw_k$ can be computed by
\begin{align*}
\|\vw_k\|^2 &= \|\vv_k\|^2 - 2 \eta_j \ip{\vv_k}{\sum_{i=1}^j \vv_i} + \eta_j^2 \ip{\sum_{i=1}^j \vv_i}{\sum_{i=1}^j \vv_i}\\
&= a - 2 \eta_j b j + \eta_j^2\left( a j + 2 b \binom{j}{2} \right) = a - \eta_j b j = \frac{(a-b)(a + b j)}{a + b(j-1)} = \tilde a.
\end{align*}
For the mixed scalar products, we observe that for every $j+1 \leq k < \ell \leq n$,
\begin{align*}
\|\vw_k - \vw_\ell\|^2 &= \|\vv_k - \vv_\ell\|^2 = 2 (a - b), \text{ and}\\
\|\vw_k - \vw_\ell\|^2 &= \|\vw_k\|^2 + \|\vw_\ell\|^2 - 2 \ip{\vw_k}{\vw_\ell}  = 2 \, \|\vw_k\|^2 - 2 \ip{\vw_k}{\vw_\ell}.
\end{align*}
Thus,
\[
\ip{\vw_k}{\vw_\ell} = \|\vw_k\|^2 - (a-b) = \tilde a - (a - b) = \tilde b,
\]
as desired.
\end{proof}

We can now discuss consequences for the packing minima of~$B_n$ with respect to the lattices~$\Lambda_{a,b}$.
In particular, these lattices contain many examples for which the sequence of packing minima is not decreasing (compare the initial discussion for the boxes~$C(\vr)$ in the introduction).
Moreover, they provide evidence for the conjectured transference bound in
Conjecture~\ref{conj:rho-transference}.

\begin{corollary}
\label{cor:Lambda-ab-bounds}
Let $a > 0$ and $b \in \R$ be such that
\[
a \geq 2 b \geq 0 \qquad \text{or} \qquad a \geq -n b \geq 0.
\]
Then, for every $1 \leq j \leq n$, we have
\[
\rho_j(B_n,\Lambda_{a,b}) \geq \frac12 \sqrt{\frac{(a-b)(a+b(n-j))}{a+b(n-j-1)}},
\]
and
\[
\rho_j(B_n,\Lambda_{a,b}) \cdot \rho_{n-j+1}(\dual{(B_n)_c},\dual{\Lambda_{a,b}}) \geq \frac12.
\]
Moreover, if $b \neq 0$, then
\[
\rho_{n-1}(B_n,\Lambda_{a,b}) < \rho_n(B_n,\Lambda_{a,b}).
\]
\end{corollary}

\begin{proof}
The first inequality is immediate from Proposition~\ref{prop:ex-ball-symmetric-lattice}~\ref{item:Lambda-ab-sucmins}, because by definition $\rho_j(B_n,\Lambda_{a,b}) \geq \lambda((B_n)_c|L^\perp,\Lambda_{a,b}|L^\perp) = \frac12 \lambda_1(B_n|L^\perp,\Lambda_{a,b}|L^\perp)$, for every $L \in \grass(\Lambda_{a,b},n-j)$.

For the second statement we use Lemma~\ref{lem:props_lambda-ab} and its notation to get $\dual{\Lambda_{a,b}} = \Lambda_{\bar a,\bar b}$.
Further, one quickly checks that $a \geq 2b \geq 0$ if and only if $\bar a \geq -n \bar b \geq 0$, and similarly, $a \geq -n b \geq 0$ if and only if $\bar a \geq 2 \bar b \geq 0$.
We can thus apply Proposition~\ref{prop:ex-ball-symmetric-lattice}~\ref{item:Lambda-ab-sucmins} as above and get the bound
\[
\rho_{n-j+1}(\dual{(B_n)_c},\dual{\Lambda_{a,b}})^2 = 4 \cdot \rho_{n-j+1}(B_n,\Lambda_{\bar a,\bar b})^2 \geq \frac{(\bar a - \bar b)(\bar a + \bar b(j-1))}{\bar a + \bar b(j-2)}.
\]
Together with the lower bound on $\rho_j(B_n,\Lambda_{a,b})$ and the definitions of $\bar a,\bar b, a'$ and $b'$ in Lemma~\ref{lem:props_lambda-ab}, we obtain the claimed estimate
\begin{align*}
&4 \cdot \rho_j(B_n,\Lambda_{a,b})^2 \cdot \rho_{n-j+1}(\dual{(B_n)_c},\dual{\Lambda_{a,b}})^2 \\
&\geq \frac{(a-b)(a+b(n-j))}{a+b(n-j-1)} \cdot \frac{1}{aa'+(n-1)bb'} \cdot \frac{(a' - b')(a' + b'(j-1))}{a' + b'(j-2)}\\
&= \frac{(a-b)(a+b(n-j))}{a+b(n-j-1)} \cdot \frac{(a + b(n-1))(a + b(n-j-1))}{(aa'+(n-1)bb')(a + b(n-j))}\\
&= \frac{(a-b)(a + b(n-1))}{aa'+(n-1)bb'} = \frac{(a-b)(a + b(n-1))}{a^2+ab(n-2)-b^2(n-1)} = 1.
\end{align*}
Now assume that $b \neq 0$.
We know that $\rho_n(B_n,\Lambda_{a,b}) = \lambda_1((B_n)_c,\Lambda_{a,b}) = \sqrt{a}/2$ and $\rho_{n-1}(B_n,\Lambda_{a,b}) = \max\{\lambda_1(B_n|\vw^\perp,\Lambda_{a,b}|\vw^\perp) : \vw \in \Lambda_{a,b} \setminus \{\vnull\}\}/2$.
Since the $\vv_1,\ldots,\vv_n$ are linearly independent and not pairwise orthogonal, for every $\vw \in \Lambda_{a,b} \setminus \{\vnull\}$, at least one of the projections $\vv_k|\vw^\perp$, $1 \leq k \leq n$, does not vanish and is shorter than $\vv_k$.
Hence, $\lambda_1(B_n|\vw^\perp,\Lambda_{a,b}|\vw^\perp) < \sqrt{a}$.
By Lemma~\ref{lem:rho-i-attained} the packing minima are attained, and thus we indeed get the strict inequality $\rho_{n-1}(B_n,\Lambda_{a,b}) < \sqrt{a}/2 = \rho_n(B_n,\Lambda_{a,b})$.
\end{proof}

We finish our considerations of the lattices $\Lambda_{a,b}$ by investigating the special case $\Lambda_n := \Lambda_{1,\frac12} \in \Lat^n$ of the regular simplex lattice a bit more in detail.
We determine $\rho_{n-1}(B_n,\Lambda_n)$ in small dimensions, and illustrate the complications that arise by computing the packing minima exactly in higher dimensions and for other indices.

Our general conjecture is that for the lattice $\Lambda_n$ the lower bounds on the packing minima in Corollary~\ref{cor:Lambda-ab-bounds} are actually attained, that is,
\begin{align}
\rho_j(B_n,\Lambda_n) = \frac12\sqrt{\frac{n-j+2}{2(n-j+1)}},\label{eq:rho-j-reg-simpl-conj}
\end{align}
for $1 \leq j \leq n$.
This means that $\rho_j(B_n,\Lambda_n)$ should be attained by the projection along $L_{n-j} = \lin\{\vv_1,\ldots,\vv_{n-j}\}$, where as usual $\{\vv_1,\ldots,\vv_n\}$ is a $(1,\frac12)$-equiangular basis.

We are able to confirm these speculations in a restricted setting:

\begin{proposition}
Let $\Lambda_n$ be the regular simplex lattice as above. Then,
\[
\rho_{n-1}(B_n,\Lambda_n) = \frac{\sqrt{3}}{4},\quad\text{ for } n \leq 7.
\]
\end{proposition}
\begin{proof}
The lower bound $\rho_{n-1}(B_n,\Lambda_n) \geq \sqrt{3}/4$ holds in every dimension~$n$ by Corollary~\ref{cor:Lambda-ab-bounds}.
For the upper bound, we invoke~\eqref{eq:def-packing-minima} and see that we have to show that $\lambda_1((B_n)_c|\vz^\perp,\Lambda_n|\vz^\perp) = \frac12\lambda_1(B_n|\vz^\perp,\Lambda_n|\vz^\perp) \leq \sqrt{3}/4$, for every $\vz \in \Lambda_n \setminus \{\vnull\}$.
In other words, we need to find for every $\vz \in \Lambda_n \setminus \{\vnull\}$ a non-zero vector in the projected lattice~$\Lambda_n|\vz^\perp$ whose squared length is at most $3/4$.
To this end, let $\vy \in \Lambda_n \setminus \{\vnull\}$ and let
\[
\vy' = \vy|\vz^\perp = \vy - \frac{\ip{\vy}{\vz}}{\ip{\vz}{\vz}} \vz.
\]
We thus want that $\|\vy'\|^2 = \|\vy\|^2 - \ip{\vy}{\vz}^2/\|\vz\|^2 \leq 3/4$.
Getting rid of the quadratic nature of this inequality, we choose the vector $\vy \in \Lambda_n$ to be of length one, that is,
\begin{align}
\vy \in \{\pm(\vv_k - \vv_\ell) : 1 \leq k < \ell \leq n\} \cup \{\pm\vv_1,\ldots,\pm\vv_n\} =: S(\Lambda_n).\label{eq:choice-of-y}
\end{align}
Then, for each $\vz \in \Lambda_n \setminus \{\vnull\}$, we need to find some~$\vy \in S(\Lambda_n)$ such that
\[
\ip{\vy}{\vz} \geq \frac12 \|\vz\|\quad\text{ and }\quad \vy \notin \lin\{\vz\}.
\]
The second condition guarantees that the projection~$\vy'$ is indeed non-zero.
However, if $\vy \in \lin\{\vz\}$, for some $\vy \in S(\Lambda_n)$, then, of course, $\vz$ is parallel to~$\vy$.
Let us first assume that $\vz = \vv_k$, for some $1 \leq k \leq n$.
Up to relabeling we are thus projecting along the line $L_1$ in Proposition~\ref{prop:ex-ball-symmetric-lattice}~\ref{item:Lambda-ab-sucmins} and get the desired bound.
If $\vz = \vv_k - \vv_\ell$ for some $k \neq \ell$, then for $\vy = \vv_k$, we have
\[
\ip{\vy}{\vz} = \ip{\vv_k}{\vv_k - \vv_\ell} = 1 - \frac12 = \frac12 \|\vz\|,
\]
so that we are fine as well.

Therefore, we can neglect the condition $\vy \notin\lin\{\vz\}$ and thus reduced the problem to the following:
\[
\text{For every } \vz \in \Lambda_n \setminus \{\vnull\} \text{ find } \vy \in S(\Lambda_n) \text{ such that} \ip{\vy}{\vz} \geq \frac12 \|\vz\|.
\]
Scaling $\vz$ to unit length, it thus suffices to prove that  the minimal $\mu \geq 0$ such that there is a $\vz \in \R^n$ with $\|\vz\|=1$ and $\ip{\vy}{\vz} \leq \mu$, for all $\vy \in S(\Lambda_n)$, equals $1/2$.
In other words, the origin-symmetric polytope
\[
P_n := \left\{\vx \in \R^n : \ip{\vx}{\vy} \leq 1, \text{ for all }\vy \in S(\Lambda_n)\right\}
\]
has circumradius $\cR(P_n) \leq 2$.
Gritzmann \& Klee~\cite{gritzmannklee1992inner} showed that for origin-symmetric convex bodies the circumradius is dual to the inradius, that is, $\ir(\dual{P_n})\cR(P_n)=1$.
Hence, the claim above is equivalent to $\ir(\dual{P_n}) \geq 1/2$.
By basic convex geometry we find that $\dual{P_n} = \conv(S(\Lambda_n))$, which is equal to the difference body $S_n-S_n$ of the regular simplex~$S_n \in \Kn$ of edge-length one.
Rogers \& Shephard~\cite{rogersshephard1957the} investigated $S_n-S_n$ in detail:
\begin{enumerate}[label=\roman*)]
 \item Every facet of $S_n-S_n$ is of the form $S_r - S_{n-r-1}$, for some $0 \leq r \leq n-1$, and where $S_r$ is an $r$-dimensional face of~$S_n$, and $S_{n-r-1}$ is its opposite face of dimension $n-r-1$.
 \item The height $h_r$ of such a facet $S_r - S_{n-r-1}$ over the origin~$\vnull$ is the distance of the barycenters of $S_r$ and $S_{n-r-1}$, and equals
 \begin{align}
 h_r = \sqrt{\frac{n+1}{2(r+1)(n-r)}}.\label{eq:height-hr}
 \end{align}
\end{enumerate}
Now, due to the regularity of~$S_n$ and $S_n-S_n$, the inradius $\ir(S_n-S_n)=\ir(\dual{P_n})$ equals the minimal such height~$h_r$, $0 \leq r \leq n-1$.
Since for $r=\lfloor n/2 \rfloor$ the value of~$h_r$ in~\eqref{eq:height-hr} grows as $\sqrt{2/n}$, we see that $\ir(\dual{P_n}) \geq 1/2$ is not true in large dimensions.
However, it does hold for every $n \leq 7$ and thus, summarizing our considerations above, we have
\[
\rho_{n-1}(B_n,\Lambda_n) = \frac{\sqrt{3}}{4},\quad\text{ for } n \leq 7.\qedhere
\]
\end{proof}

\noindent The arguments in the proof above are based on the strong assumption that we only consider vectors $\vy \in S(\Lambda_n)$.
So, for computing $\rho_{n-1}(B_n,\Lambda_n)$ in higher dimensions, we would need to project longer lattice vectors from~$\Lambda_n$.

\bibliographystyle{amsplain}
\bibliography{mybib}

\end{document}